\documentclass{amsart}

\usepackage[utf8]{inputenc}
\usepackage[T1]{fontenc}
\usepackage{float}

\usepackage{xfrac}

\usepackage{bm}

\usepackage{baskervillef}
\usepackage[baskerville, varg]{newtxmath}
\usepackage[scr=boondox,cal=boondoxo,bb=boondox,frak=euler]{mathalfa}
\usepackage{xcolor}
\definecolor{darkred}{rgb}{0.4,0,0}
\definecolor{darkgreen}{rgb}{0,0.4,0}
\definecolor{darkblue}{rgb}{0,0,0.4}

\definecolor{darkgray}{rgb}{0.55,0.55,0.55}
\definecolor{lightgray}{rgb}{0.9,0.9,0.9}

\usepackage[colorlinks,citecolor=darkgreen,linkcolor=darkred,urlcolor=darkblue,filecolor=darkred]{hyperref}
\usepackage[noabbrev, nameinlink]{cleveref}
\makeatletter
\def\@cite#1#2{{\m@th\upshape\bfseries%
[{#1\if@tempswa{\m@th\upshape\mdseries, #2}\fi}]}}
\makeatother
\usepackage{subcaption}

\usepackage[style=alphabetic,natbib=true,giveninits=true,url=false,maxbibnames=99]{biblatex}

\DeclareFieldFormat{title}{\mkbibquote{#1}}
\DeclareFieldFormat[misc]{date}{Preprint (#1)}
\AtEveryBibitem{\clearfield{issn}}
\AtEveryCitekey{\clearfield{issn}}

\usepackage{graphicx}
\usepackage{multirow}
\usepackage{enumitem}
\usepackage{todonotes}

\allowdisplaybreaks[4]

\renewcommand{\H}{\mathcal{H}}
\newcommand{\M}{\mathcal{M}}
\renewcommand{\O}{\mathcal{O}}

\newcommand{\CC}{\mathbb{C}}

\newcommand{\NN}{\mathbb{N}}
\newcommand{\QQ}{\mathbb{Q}}
\newcommand{\RR}{\mathbb{R}}
\newcommand{\ZZ}{\mathbb{Z}}

\newcommand{\modelL}{\mathrm{L}}
\newcommand{\modelP}{\mathrm{P}}

\newcommand{\tr}{{\mbox{\raisebox{0.33ex}{\scalebox{0.6}{$\intercal$}}}}}

\renewcommand{\Re}{\mathrm{Re}}

\newcommand{\Jac}{\mathrm{Jac}}
\newcommand{\Aut}{\mathrm{Aut}}
\newcommand{\End}{\mathrm{End}}
\newcommand{\Aff}{\text{Aff}} 

\newcommand{\SL}{\mathrm{SL}}

\newcommand{\GL}{\mathrm{GL}}
\newcommand{\GA}{\mathrm{GA}}

\newcommand{\Sym}{\mathrm{Sym}}
\newcommand{\Dih}{\mathrm{Dih}}

\newcommand{\DihFour}[4]{\ensuremath{\mathop{\mathrm{Dih}}\Biggl(\!\!
\raisebox{-.4\height}{
\begin{tikzpicture}[scale=.6,every node/.style={scale=.75,draw=none}]
\node (V1) at (-{(1+3)*360/4+90}:1) {\ensuremath{#1}};
\node (V2) at (-{(2+3)*360/4+90}:1) {\ensuremath{#2}};
\node (V3) at (-{(3+3)*360/4+90}:1) {\ensuremath{#3}};
\node (V4) at (-{(4+3)*360/4+90}:1) {\ensuremath{#4}};
\draw[thick] (V1) -- (V2) -- (V3) -- (V4) -- (V1);
\end{tikzpicture}
}\!\!\Biggr)
}}
\newcommand{\DihFive}[5]{\ensuremath{\mathop{\mathrm{Dih}}\Biggl(\!
\raisebox{-.4\height}{
\begin{tikzpicture}[scale=.6,every node/.style={scale=.75,draw=none}]
\node (V1) at (-{(1+4)*360/5+90}:1) {\ensuremath{#1}};
\node (V2) at (-{(2+4)*360/5+90}:1) {\ensuremath{#2}};
\node (V3) at (-{(3+4)*360/5+90}:1) {\ensuremath{#3}};
\node (V4) at (-{(4+4)*360/5+90}:1) {\ensuremath{#4}};
\node (V5) at (-{(5+4)*360/5+90}:1) {\ensuremath{#5}};
\draw[thick] (V1) -- (V2) -- (V3) -- (V4) -- (V5) -- (V1);
\end{tikzpicture}
}\!\Biggr)
}}
\newcommand{\fr}{\mathrm{fr}}

\newcommand{\modn}[1]{\mathbin{\equiv_{#1}}}
\newcommand{\nmodn}[1]{\mathbin{{\not\equiv}_{#1}}}

\newtheorem{thm}{Theorem}[section]

\newtheorem{lem}[thm]{Lemma}
\crefname{lem}{lemma}{lemmas}
\newtheorem{prop}[thm]{Proposition}
\crefname{prop}{proposition}{propositions}
\newtheorem{cor}[thm]{Corollary}
\theoremstyle{definition}
\newtheorem{defn}[thm]{Definition}
\theoremstyle{remark}
\newtheorem{rem}[thm]{Remark}

\DeclareFieldFormat{title}{\mkbibquote{#1}}

\makeatletter
\g@addto@macro\bfseries{\boldmath}
\makeatother

\begin{document}

	\title[Permutation of periodic points on $\H(2)$]{Permutation of periodic points of Veech surfaces in $\H(2)$}
	
	\author[R. Gutiérrez-Romo]{Rodolfo Gutiérrez-Romo}
	\address[Rodolfo Gutiérrez-Romo]{Centro de Modelamiento Matemático, CNRS-IRL 2807, Universidad de Chile, Beauchef 851, Santiago, Chile.}
	\email{g-r@rodol.fo}
	\urladdr{http://rodol.fo}
	
	\author[A. Pardo]{Angel Pardo}
	\address[Angel Pardo]{Departamento de Matemática y Ciencia de la Computación, Universidad de Santiago de Chile, Las Sophoras 173, Estación Central, Santiago, Chile.
}
	\email{angel.pardo@usach.cl (corresponding author)}
	
	
	
	\thanks{This work was supported by the Center for Mathematical Modeling (CMM), ACE210010 and FB210005, BASAL funds for centers of excellence from ANID-Chile, and the MATH-AmSud 21-MATH-07 grant.
	The first named author was also supported by ANID-Chile through the FONDECYT Iniciación 11190034 grant, while the second named author was also supported by ANID-Chile through the FONDECYT Postdoctorado 3190257 and FONDECYT Regular 1221934 grants.}

	\begin{abstract}
	    We study how are permuted Weierstrass points of Veech surfaces in $\H(2)$, the stratum of Abelian differentials on Riemann surfaces in genus two with a single zero of order two.
	    These surfaces were classified by McMullen relying on two invariants: discriminant and spin.
    	More precisely, given a Veech surface in $\H(2)$ of discriminant $D$, we show that the permutation group induced by the affine group on the set of Weierstrass points is isomorphic to $\Dih_4$, if $D \modn{4} 0$; to $\Dih_5$, if $D \modn{8} 5$; and to $\Dih_6$, if $D \modn{8} 1$.
        Moreover, these same groups arise when considering only Dehn multitwists of the affine group.
	\end{abstract}
	
	\maketitle
 	\setcounter{tocdepth}{1}

\section{Introduction}
	In this work, we study how the Weierstrass points of a genus two flat surface are permuted by affine transformations.
	More precisely, these flat surfaces correspond to \emph{translation surfaces} or, equivalently, Abelian differentials on Riemann surfaces, and have a natural flat metric with conical singularities at the zeros of the differential.
	Furthermore, their moduli spaces are stratified according to the combinatorics of zeros of the corresponding differentials.
    There is a natural action of $\SL(2, \RR)$ on each stratum, which generalizes the action of $\SL(2, \RR)$ on the space $\GL(2, \RR)/\SL(2, \ZZ)$ of flat tori. The dynamical, algebraic and geometric properties of this action have been extensively studied. We refer the reader to the survey by Zorich for more details \cite{Zorich:survey}.
    
    Given a translation surface $X$, its stabilizer in $\SL(2, \RR)$ is called the \emph{Veech group} of $X$ and is denoted $\SL(X)$. If this group is a lattice, $X$ is said to be a \emph{Veech surface}. The derivative map $D\colon \Aff(X) \to \SL(X)$ defines a finite-to-one group homomorphism, where $\Aff(X)$ denotes the group of affine diffeomorphisms of $X$, that is, diffeomorphisms whose local expressions are affine with respect to the flat metric.
    
    When $X$ is a Veech surface, we say that a point $x \in X$ is \emph{periodic} if its $\Aff(X)$-orbit is finite. Periodic points are significant to various questions in billiard dynamics, such as the illumination or finite blocking problems. Hence, they have attracted a great deal of attention during recent years.
    
    In this context, our goal is to study the following question: how can periodic points in a Veech surface be permuted by the affine group?
    
    In general, explicitly determining the periodic points of a surface has proven to be very difficult. In some cases, however, the answer is known. In particular, this holds for Veech surfaces in genus $2$, and special kinds of Veech surfaces in genus $3$ and $4$ (known as \emph{Weierstrass Prym eigenforms}). This classification depends on whether the translation surface is \emph{primitive}, which is defined as not being the pullback of a holomorphic one-form on a lower-genus Riemann surface under a holomorphic branched cover. Möller~\cite{Moeller} showed that the periodic points on primitive Veech surfaces in genus two coincide with the Weierstrass points, and this was very recently extended to primitive Weierstrass Prym eigenforms in genus $3$ and $4$ by Freedman \cite{Freedman}. For non-primitive such Veech surfaces, the set of periodic points is dense and it is also known by these works.
    
    In the case of genus two, the moduli space of translation surfaces has two connected strata, namely $\H(2)$ and $\H(1,1)$, consisting of translation surfaces with a single conical singularity of angle $6\pi$, or two conical singularities of angle $4\pi$, respectively.
    Veech surfaces of genus two were classified by McMullen~\cite{McMullen:classification} in terms of two invariants: a real quadratic \emph{discriminant} $D \geq 5$, together with a \emph{spin} invariant when $D \modn{8} 1$.

    In the stratum $\H(1,1)$,
    McMullen~\cite{McMullen:decagon} shows that the only primitive Veech surfaces are given by (affine deformations of) the regular decagon surface. It is possible to show that the group of permutations obtained by the action of the affine group on its six Weierstrass points is isomorphic to $\Dih_5$, the dihedral group of order~$10$.
    Note that this case corresponds to discriminant $D = 5$.
    
    In the case of Veech surfaces in the stratum $\H(2)$, the singularity is always a Weierstrass point and is a fixed point for the action of the affine group.
    Thus, we are interested on the action on the \emph{regular} Weierstrass points, that is, Weierstrass points that are regular for the flat metric, and we prove the following.
    
	\begin{thm}
		Let $X \in \H(2)$ be a Veech surface of discriminant $D \geq 9$ and let $w_1, \dotsc, w_5$ be its regular Weierstrass points. Let $G(X) \leqslant \Sym(\{w_1, \dotsc, w_5\})$ be the permutation group induced by the action of $\Aff(X)$ on $\{w_1, \dotsc, w_5\}$. Then, $G(X)$ is isomorphic to:
		\begin{itemize}
			\item $\Dih_4$ if $D \modn{4} 0$;
			\item $\Dih_5$ if $D \modn{8} 5$; and
			\item $\Dih_6$ if $D \modn{8} 1$, regardless of the spin invariant;
		\end{itemize}
		where $\Dih_n$ is the dihedral group of order $2n$.
		\label{thm:main_H(2)}
	\end{thm}
	
    Note that the previous theorem implies that the isomorphism class of $G(X)$ is a ``weak'' invariant: it allows us to deduce the residue class of the discriminant $D$, but not the spin invariant. Furthermore, the classification does not depend on whether the Veech surface is primitive or not, although the proofs are done separately for these two cases.

    	\subsubsection*{The HLK-invariant}
        	In the non-primitive case, a Veech surface $X \in \H(2)$ is \emph{arithmetic}, that is, it corresponds to a branched cover of a flat torus (elliptic curve), branched over the $1$-torsion point.
    	    It follows that the regular Weierstrass points on $X$ project down to $2$-torsion points on the torus. Moreover, any element of $\SL(2, \RR)$ preserves both the sets of $1$-torsion points and of $2$-torsion points. These observations give rise to an invariant for the $\SL(2, \RR)$-orbit of $X$, defined as the combinatorics of these projections on the covered torus; it is known as \emph{HLK-invariant} since it was first considered by Kani~\cite{Kani} and Hubert--Lelièvre~\cite{Hubert-Lelievre} (see also the work by Matheus--Möller--Yoccoz \cite{Matheus-Moeller-Yoccoz}).
    	    This invariant naturally restricts the action of $\SL(X)$ on the set of regular Weierstrass points.
    	    In this context, \Cref{thm:main_H(2)} shows that in fact this is the only restriction.
    	    
        	Moreover, in the non-arithmetic case, two distinct situations can arise. If the discriminant $D$ is a quadratic residue modulo~$8$, a suitable reduction allows us to apply the same criterion even if the periods are no longer integral (they belong to a real quadratic field).
        	This is no longer true when $D$ is a quadratic nonresidue modulo $8$ and, thus, we are not able to extend this argument for $D \modn{8} 5$.
            However, we are still able to adapt these ideas to constrain the action of the affine group on regular Weierstrass points and show that these restrictions are, as well, the only ones.
    
        \subsubsection*{Generators}
            A natural follow-up question is which elements of the affine group are required to generate the group $G(X)$ as in \Cref{thm:main_H(2)}.
            In our proof, we use affine Denh multitwists, that are in correspondence with parabolic elements in the Veech group.
            As a consequence of our proof of \Cref{thm:main_H(2)} we obtain the following.
        	\begin{cor}\label{cor:main_H(2)}
        		For any discriminant $D > 4$ with $D \notin \{9,33\}$, the group $G(X)$ can be generated by the action of two affine multitwists in $\Aff(X)$. If $D \in \{9,33\}$, nevertheless, the minimal number of affine multitwists that are needed to generate the group $G(X)$ is three.
        	\end{cor}
	
	\subsection*{Related results}
        Periodic points arise naturally in the study of Veech surfaces. One example of this fact is that the Teichmüller curve induced by the $\SL(2, \RR)$-orbit of $X$ is non-arithmetic if and only if $X$ possesses finitely many periodic points \cite{Gutkin-Judge,Gutkin-Hubert-Schmidt}.
        They also naturally arise in fundamental problems from billiard dynamics, such as the finite blocking and the illumination problems. See for example the seminal works of Monteil~\cite{Monteil} and Lelièvre--Monteil--Weiss~\cite{Lelievre-Monteil-Weiss} that describe these connections.
        
        Furthermore, since the classical notion of periodic point is not suitable for a non-Veech surface $X$ (as $\Aff(X)$ may be trivial), it was extended for such surfaces by Apisa~\cite{Apisa:marked}. From the work of Eskin--Filip--A.~Wright~\cite{Eskin-Filip-Wright}, one can deduce the corresponding finiteness result for any non-arithmetic translation surface.

        Recently, there has been a flourishing interest in computing the exact set of periodic points.
        In fact, the finite set of periodic points for some known non-arithmetic surfaces have been computed by Möller~\cite{Moeller}, for Veech surfaces in genus~2; by Apisa~\cite{Apisa:marked,Apisa}, for non-arithmetic Prym eigenforms in genus~2, and for generic surfaces in any component of a strata of translation surfaces; by Apisa--Saavedra--Zhang~\cite{Apisa-Saavedra-Zhang}, for regular $2n$-gons and double $(2n+1)$-gons; and by B.~Wright~\cite{Wright:Ward-Veech}, for some Ward Veech surfaces. They all correspond to the set of fixed points of some (affine) involution (with derivative $-\mathrm{Id}$).

        Our results complement the knowledge of periodic points and their dynamics in the case of Veech surfaces in genus two and allows to give more precise information on related problems as, for example, the illumination and finite blocking problems.
        Our techniques may be extended to study the permutation groups in the other known examples.

	\subsection*{Further work}
        Veech surfaces in $\H(2)$ naturally form part of a larger family of Veech surfaces, first discovered and studied by McMullen~\cite{McMullen:Prym} and then classified by Lanneau--Nguyen~\cite{Lanneau-Nguyen:H4,Lanneau-Nguyen:H6}, known as \emph{Weierstrass Prym eigenforms}.
        Very recently, Freedman~\cite{Freedman} classified the periodic points on primitive Weierstrass Prym eigenforms, showing that there are no more such points than previously expected.
        In a forthcoming paper~\cite{Gutierrez-Romo--Pardo}, we adapt the ideas of the present work and classify the permutation group induced by the affine group on their periodic points in a similar fashion.
	
    \subsection*{Strategy of the proof}
        McMullen's classification~\cite{McMullen:discriminant_spin} of Teichmüller curves in $\H(2)$ relies on a combinatorial description of two-cylinder cusps, called \emph{splitting prototypes}.
        This gives rise to \emph{prototypical surfaces} on each orbit of a Veech surface. Since, for any Veech surface $X$, the permutation groups $G(X)$ and $G(g \cdot X)$ are conjugate for every $g \in \SL(2, \RR)$, it is enough to compute $G(X)$ for prototypical surfaces.
        
        Given a prototypical surface $P$, we start by computing the action of some parabolic elements of $\SL(P)$ on its regular Weierstrass points. This shows that a specific dihedral group $H$, depending only on the discriminant, is contained inside $G(P)$. Some cases where $P$ has small discriminant, namely $9$ or $33$, have to be treated separately, as the general argument is insufficient in these particular cases.
        
        Then, we use geometric restrictions imposed by the locations of the Weierstrass points, and algebraic restrictions imposed by the coefficients of $\SL(P)$, to show that, in fact, $G(P)$ is a subgroup of the dihedral group $H$.
        In the arithmetic case, these restrictions are those given by the HLK-invariant. In the non-arithmetic case, we are able to extend the HLK-invariant when the discriminant $D$ is a quadratic residue modulo $8$ and obtain the desired reverse inclusion in this way.
        For the remaining case of $D \modn{8} 5$, we show the reverse containment by analysing the same geometric and algebraic restrictions. However, we do \emph{not} find any natural extension of the HKL-invariant for this latter case.
        
        By combining both inclusions, we conclude the desired results proving \Cref{thm:main_H(2)} and \Cref{cor:main_H(2)}.

    \subsection*{Structure of the paper}
        In \Cref{sec:background}, we introduce the general background necessary to formulate and prove our results. 
        In \Cref{sec:preliminaires_H(2)}, we recall McMullen's classification of Teichmüller curves in $\H(2)$, and introduce splitting prototypes and prototypical surfaces. \Cref{sec:parabolic_H(2)} is devoted to the study of the action of parabolic elements on regular Weierstrass points: we compute their action on prototypical surfaces, showing that the permutation group given by action of the Veech group on regular Weierstrass points contains a subgroup isomorphic to an appropriate dihedral group. In \Cref{sec:upper-bounds_H(2)}, we compute the corresponding reverse inclusions; this section is divided into several subsections. Indeed, in \Cref{sec:arithmetic_H(2)} we study the arithmetic case and compute the distribution of regular Weierstrass points over $2$-torsion in the unit torus, giving the corresponding restrictions for the action on regular Weierstrass points.
        In \Cref{sec:non-arithmetic_H(2)}, we study the non-arithmetic case; the case of quadratic residues modulo~$8$ is treated in \Cref{sec:quadratic-residues_H(2)} and, in \Cref{sec:quadratic-nonresidues_H(2)}, we study the case of $D \modn{8} 5$.
        
        For some small discriminants the general arguments are insufficient; we treat these exceptional cases separately in \Cref{sec:remaining_H(2)}.
        Finally, in \Cref{sec:proof_H(2)}, we summarise our results and give a proof of \Cref{thm:main_H(2)} and \Cref{cor:main_H(2)}.
    
	\subsection*{Acknowledgements}
	    We are grateful to Vincent Delecroix and Carlos Matheus for their interesting comments, which motivated some parts of this work.
	
\section{Background}
\label{sec:background}
	A \emph{translation surface} is a pair $(X, \omega)$, where $X$ is a Riemann surface and $\omega$ is a nonzero Abelian differential on $X$. If $\omega(p) = 0$ for $p \in X$, we say that $p$ is a \emph{zero} or \emph{singularity}. The \emph{order} of $p$ is the order of $p$ viewed as a zero of $\omega$. On the other hand, if $\omega(p) \neq 0$, we say that $p$ is a \emph{regular point}. By a slight abuse of notation, we will often omit the differential and say that $X$ is a translation surface. For an introduction and more details on this subject, we refer the reader to Zorich's survey~\cite{Zorich:survey}.
	
	The moduli space of translation surfaces is stratified by the list of orders of zeros of the Abelian differential: the \emph{stratum} $\H(\kappa_1, \dotsc, \kappa_n)$ consists of translation surfaces whose singularities have prescribed orders $\kappa_1, \dotsc, \kappa_n$ (with multiplicities). The genus $g \geq 1$ of $X$ can be deduced from the combinatorial data by the Gauss--Bonnet-like formula $\kappa_1 + \dotsb + \kappa_n = 2g - 2$. Strata are endowed with the natural topology coming from period coordinates \cite[Section~3]{Zorich:survey}.

	A translation surface may also be defined as an equivalence class of collections of polygons in $\CC$ in which every side is identified with another parallel, equal-length, opposite side by translation. The equivalence relation is the one induced by \emph{cut-and-paste operations}: two such collections represent the same translation surface if it is possible to cut pieces of the former along straight lines and glue them using the side identifications in order to obtain the latter. This polygon definition induces a \emph{flat metric} on the translation surface which becomes singular at the zeros. Indeed, an order-$k$ zero of the Abelian differential corresponds to a singular point for the flat metric, with cone angle $2\pi(1 + k)$.
	
	There exists a natural action of $\SL(2, \RR)$ on a stratum given by linearly deforming the collection of polygons (by using $\CC \cong \RR^2$). With the first definition of translation surface, this action is exactly post-composition with coordinate charts (again using $\CC \cong \RR^2$). The $\SL(2, \RR)$-action preserves both the stratum in which a translation surface $X$ lies and the total area of $X$. Hence, it induces an action on the set $\H^{(r)}(\kappa) \subseteq \H(\kappa)$ of area-$r$ translation surfaces inside a given stratum $\H(\kappa)$, for some fixed $r > 0$.
	
	Let $X \in \H^{(r)}(\kappa)$ be an area-$r$ translation surface. Regarding its orbit by the $\SL(2, \RR)$-action, two extreme cases can arise:
	\begin{itemize}
		\item $X$ is a \emph{Veech surface} (or simply \emph{Veech}), that is, its $\SL(2, \RR)$-orbit is closed inside $\H^{(r)}(\kappa)$; or
		\item $X$ is \emph{generic}, that is, its $\SL(2, \RR)$-orbit is dense in $\mathcal{C}$, where $\mathcal{C}$ is the connected component of $\H^{(r)}(\kappa)$ containing $X$.
	\end{itemize}
	
	The stabiliser $\SL(X)$ of $X$ for the $\SL(2, \RR)$-action is called the \emph{Veech group} of $X$. We have that $X$ is a Veech surface if and only if $\SL(X)$ is a lattice inside $\SL(2, \RR)$, that is, $\SL(2, \RR)/\SL(X)$ has finite measure\footnote{Veech surfaces were originally defined as those whose Veech group is a lattice. The fact that this is indeed equivalent to having a closed $\SL(2,\RR)$-orbit is known as \emph{Smillie's theorem}, since it was originally proved, though not published, by John Smillie; it was first referenced in Veech's work \cite{Veech:Smillie}.}. For this reason, Veech surfaces are also often called \emph{lattice surfaces}.
	
	The projection of the $\SL(2, \RR)$-orbit of a Veech surface into the moduli space of Riemann surfaces is called a \emph{Teichmüller curve}. They come in two different flavours: 
	\begin{itemize}
		\item the \emph{arithmetic} ones, which arise as the projection of the $\SL(2, \RR)$-orbit of a \emph{square-tiled surface}, that is, a translation surface that admits a covering of the unit torus branched over a single point; and
		\item \emph{non-arithmetic} ones, which cannot be defined in terms of a square-tiled surface.
	\end{itemize}
	The distinction between arithmetic and non-arithmetic Teichmüller curves is important, as different tools are usually used for their study. Indeed, square-tiled surfaces admit purely combinatorial descriptions, unlike those belonging to non-arithmetic Teichmüller curves.
	
	The $\SL(2, \RR)$-action can be thought of as part of a larger action by $\GL^+(2, \RR)$ on a stratum; the main difference being that the latter does not necessarily preserve area. Nevertheless, scaling a translation surface will not change the aforementioned properties (such as being Veech or generic). More precisely, for any $r, r' > 0$, there exists a homeomorphism between the spaces $\H^{(r)}(\kappa)$ and $\H^{(r')}(\kappa)$ simply given by the action of a constant multiple of the identity matrix in $\GL^+(2, \RR)$. As these matrices commute with any other matrix, the homeomorphism is $\SL(2, \RR)$-equivariant. We will sometimes make use of such scalings to simplify some computations.
	
	\subsection*{Affine transformations}
	    Let $X$ be a Veech surface. We say that a bijective, orientation-preserving map $f\colon X \to X$ is an \emph{affine transformation} if its local expressions (on the charts given by the flat metric) are affine maps of $\RR^2$. We denote the group of affine transformations on $X$ by $\Aff(X)$.
	
    	The linear part of an affine transformation does not depend on the choice of flat charts. This observation gives rise to a group homomorphism $D \colon \Aff(X) \to \SL(2,\RR)$, known as the \emph{derivative map}. 
    	The \emph{Veech group} $\SL(X)$ of $X$ can be defined as the image of the derivative map, that is, $\SL(X) = D\Aff(X)$. Veech showed that the derivative map is finite-to-one; the kernel of this map is hence a finite group, which is known as the automorphism group $\Aut(X)$ of $X$. In other words, we have a short exact sequence
    	\[
    	    1 \to \Aut(X) \to \Aff(X) \to \SL(X) \to 1.
    	\]
    	
    	We are particularly interested in the case where $X$ has a single zero (that is, when $X$ belongs to a so-called \emph{minimal stratum}). It is known that, in this case, the group $\Aut(X)$ is trivial, so $D$ is one-to-one \cite[Proposition~4.4]{Hubert-Lelievre}.
        In what follows, in a slight abuse of notation, we will implicitly make use of the identification $\SL(X) \cong \Aff(X)$ for translation surfaces possessing a single zero.
	
	\subsection*{Periodic points}
    	We say that a point $x \in X$ is \emph{periodic} if its $\Aff(X)$-orbit is finite.
    	
    	In the case of $\H(2)$, periodic points in non-arithmetic Veech surfaces were proven to coincide with Weierstrass points by Möller~\cite{Moeller}.
    	
    	The main goal of this article is to study the $\Aff(X)$-action on Weierstrass points of Veech surfaces in $\H(2)$.

	\subsection*{Closed orbits in \texorpdfstring{$\H(2)$}{H(2)}}
    	We are interested in translation surfaces $X$ in the minimal stratum $\H(2)$ in genus two.
    	Every such surface is \emph{hyperelliptic}, that is, it admits an involution $\iota \colon X \to X$ satisfying $\iota^* \omega = -\omega$ and such that $X/\iota$ is (topologically) a sphere. This involution, known as the \emph{hyperelliptic involution}, has exactly six fixed points: the order-$2$ zero of $\omega$ and five regular points $w_1, \dotsc, w_5$. Collectively, these six points are known as the \emph{Weierstrass points} of $X$. Since this set of points must be permuted by the action of $\Aff(X)$, Weierstrass points are periodic. Moreover, Möller showed that no other periodic points exist for these surfaces when they induce non-arithmetic Teichmüller curves, that is, every periodic point is either a Weierstrass point or the order-2 zero in the non-arithmetic case \cite{Moeller}.
    	
    	By the work of McMullen~\cite{McMullen:classification}, every surface in $\H(2)$ is either Veech or generic. Moreover, McMullen \cite{McMullen:discriminant_spin}
    	also showed that Teichmüller curves coming from closed $\SL(2, \RR)$-orbits in $\H(2)$
    	are classified by two invariants: a positive integer $D > 4$ satisfying $D \modn{4} 0$ or $D \modn{4} 1$, known as the \emph{discriminant}, and, when $D \modn{8} 1$ and $D \neq 9$, the \emph{spin invariant}, which can be either even or odd. The \emph{discriminant} of a Veech surface is the discriminant of the Teichmüller curve that it induces.
	
	\subsection*{The action of the Veech group on Weierstrass points in \texorpdfstring{$\H(2)$}{H(2)}}
	    Let $X$ be a translation surface in $\H(2)$; we are interested in the action of $\SL(X)$ on regular Weierstrass points. More precisely, if $X \in \H(2)$, the group $\SL(X)$ acts by permutations on the set $\{w_1, \dotsc, w_5\}$. The permutation group $G(X) \leqslant \Sym(\{w_1, \dotsc, w_5\})$ is the group obtained by all such permutations. Observe that this group is only well-defined up to conjugation.
        In other words, we do not keep track the ``names'' of the regular Weierstrass points.

\section{Teichmüller curves in \texorpdfstring{$\H(2)$}{H(2)}}
\label{sec:preliminaires_H(2)}
	In this background section we recall McMullen's classification of Teichmüller curves in $\H(2)$ and \emph{splitting prototypes}, which describe two-cylinder cusps.
    
    \subsection{Classification of Teichmüller curves in \texorpdfstring{$\H(2)$}{H(2)}}
	Fix a Riemann surface $X$. Recall that $H^{1,0}(X)$ is the complex vector space of Abelian differentials on $X$. As a real vector space, $H^{1,0}(X)$ is canonically isomorphic to $H^1(X; \RR)$ via the isomorphism $\omega \mapsto \Re(\omega)$. This isomorphism also endows $H^{1,0}(X)$ with a symplectic form by pulling back the one on $H^1(X; \RR)$.
	
	The \emph{Jacobian} of $X$, denoted by $\Jac(X)$, is the complex torus $H^{1,0}(X)^*/H_1(X; \ZZ)$ endowed with the symplectic form on $\mathop{T_0} \Jac(X) = H^{1,0}(X)^*$ induced by the symplectic form on $H^{1,0}(X)$. An \emph{endomorphism} of $\Jac(X)$ is an endomorphism of the complex torus. By the isomorphism $H^{1,0}(X) \cong H^1(X; \RR)$, we can think of an endomorphism of $\Jac(X)$ as an endomorphism of $H_1(X; \ZZ)$ whose real-linear extension to $H^{1,0}(X)^*$ is complex-linear. An endomorphism of $\Jac(X)$ is said to be \emph{self-adjoint} if it is self-adjoint with respect to the symplectic form on $\mathop{T_0} \Jac(X)$.
	
	Fix an integer $D > 0$ satisfying $D \modn{4} 0$ or $D \modn{4} 1$. Such $D$ is known as a \emph{real quadratic discriminant}. Let $\O_D$ be the \emph{real quadratic order} of discriminant $D$, that is,
	\[
	    \O_D = \ZZ[\lambda] \cong \ZZ[T]/(T^2 - eT - b c),
    \]
	where $b$, $c$ and $e$ are integers satisfying $D=e^2 + 4bc$ and $\lambda = \frac{e + \sqrt{D}}{2}$.
	Note that if $D$ is odd and square-free, or if $D$ is even and $D/4$ is square-free, then $\O_D$ is the \emph{maximal} order inside $\QQ(\sqrt{D})$ and coincides with its ring of integers.
	 
	We say that an Abelian variety $A$ admits real multiplication by $\O_D$ if $\dim_\CC A = 2$ and $\O_D$ occurs as an indivisible, self-adjoint subring of $\End(A)$.
		
	With these definitions in mind, we can now define the loci in which we are interested:
	\begin{defn}
		The \emph{Weierstrass curve} $W_D$ of discriminant $D$ is the locus of Riemann surfaces $X \in \M_2$ such that:
		\begin{itemize}
			\item $\Jac(X)$ admits real multiplication by $\O_D$; and
			\item there exists an Abelian differential $\omega$ on $X$ such that:
			\begin{itemize}
				\item $(X, \omega) \in \H(2)$; and
				\item $\omega$ is an eigenform, that is, $\O_D \cdot \omega \subseteq \CC \cdot \omega$.
			\end{itemize}
		\end{itemize}
	\end{defn}
	
	By the work of
	McMullen~\cite{McMullen:discriminant_spin} (see also the articles by Calta \cite{Calta} and Hubert--Lelièvre \cite{Hubert-Lelievre}), we have the following:
	\begin{thm}
		Every irreducible component of a Weierstrass curve is a Teichmüller curve. Conversely, any Teichmüller curve coming from a translation surface in $\H(2)$ is an irreducible component of some Weierstrass curve.
		
		Let $D$ be a real quadratic discriminant. If $D \leq 4$, then the Weierstrass curve $W_D$ is empty. On the other hand, if $D \geq 5$, then $W_D$ is irreducible if $D \nmodn{8} 1$ or if $D = 9$ and, otherwise, $W_D$ can be decomposed as $W_D^0 \sqcup W_D^1$, where $W_D^0$ and $W_D^1$ are irreducible and distinguished by the spin invariant.
	\end{thm}
	
	For the definition of the spin invariant, we refer the reader to McMullen's original work \cite[Section~5]{McMullen:discriminant_spin}.
	
    Given a real quadratic discriminant $D$, the Weierstrass curve $W_D$ is a locus in the moduli space of \emph{Riemann} surfaces. As we will work with \emph{translation} surfaces, we define $\Omega W_D$ as the locus of translation surfaces in $\H(2)$ projecting down onto $W_D$. We have that each $\Omega W_D$ consists entirely on Veech surfaces. We similarly define $\Omega W_D^{\varepsilon}$ for $\varepsilon \in \{0, 1\}$ when $D > 9$ and $D \modn{8} 1$.
	
	\subsection{Prototypes}
    Following McMullen's classification \cite{McMullen:discriminant_spin}, we say that a quadruple of integers $(a, b, c, e)$ is a \emph{splitting prototype} of discriminant $D$, if
        \begin{gather*}
        D = e^2 + 4 bc, \quad
        0 \leq a < \gcd(b,c), \quad
    	c + e < b, \\
    	0 < b,c
    	\quad \text{ and } \quad
    	\gcd(a,b,c,e) = 1.
    \end{gather*}
    
    With each splitting prototype $(a, b, c, e)$, we associate a translation surface as in \Cref{fig:prototypical}, where $\lambda = \frac{e + \sqrt{D}}{2}$. We call such surface the \emph{prototypical surface} associated with the prototype $(a, b, c, e)$ and denote it by $\modelP(a,b,c,d) \in \Omega W_D$.
    
    \begin{figure}[t!]
        \centering
        \includegraphics[scale=0.8]{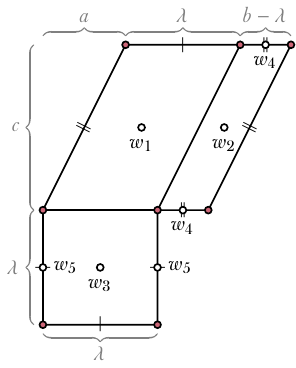}
        \caption{Prototypical surface $\modelP(a, b, c, e)$, with $\lambda = \frac{e + \sqrt{D}}{2}$. This example corresponds to $\modelP(2, 4, 4, -3)$ for $D = 73$.}
        \label{fig:prototypical}
    \end{figure}
    
    Splitting prototypes are in a one-to-one correspondence with two-cylinder cusps in $W_D$ \cite[Theorem~4.1]{McMullen:discriminant_spin} (see~\Cref{sec:parabolic-action} for more details on cylinder decompositions).
    In particular, any two cylinder decomposition of a translation surface in $\Omega W_D$ corresponds to a splitting prototype of discriminant $D$ and can be turned into a surface as in \Cref{fig:prototypical}, up to a rotation (that maps the two-cylinder direction to the horizontal direction) and the action of a diagonal matrix in $\GL^+(2,\RR)$ (in order to match the side lengths).
    We use this fact in \Cref{sec:parabolic_H(2)} to understand how parabolic elements induced by two-cylinder decompositions act on Weierstrass points.
    
    \subsubsection{The spin invariant}
    The spin invariant can be used to distinguish the two components of $\Omega W_D$ when $D > 9$ and $D \modn{8} 1$. McMullen showed \cite[Theorem~5.3]{McMullen:discriminant_spin} that the spin invariant of a splitting prototype $(a, b, c, e)$ of discriminant $D \modn{8} 1$ is given by the parity of
    \begin{equation}\label{eq:spin}
        \frac{e-f}{2} +(c+1)(a+b+ab),
    \end{equation}
    where $f > 0$ is the conductor of $\O_D$, that is, the index of $\O_D$ inside the maximal order of $\O_D\otimes \QQ \cong \QQ(\sqrt{D})$.
    In particular, when $D$ is odd, we have $D = Ef^2$ with $E$ square-free.
    
    \subsubsection{Reduced prototypes}
    A splitting prototype that has the form
    \[
        (a,b,c,e) = \left(0,\frac{D-e^2}{4},1,e\right),
    \]
    that is, such that $c=1$, is called a \emph{reduced prototype} \cite[Sections~8 and 10]{McMullen:discriminant_spin}. Given a real quadratic discriminant $D > 4$, reduced prototypes are parametrised only by $e$, where $e \in \ZZ$ satisfies
    \[
        e \modn{2} D, \quad e^2 < D \quad \text{ and } \quad (e + 2)^2 < D.
    \]
    Observe that the last two inequalities are equivalent to  $0 < \lambda < b$, where $\lambda = \frac{e + \sqrt{D}}{2}$, and can also be restated as $-\sqrt{D} < e < \sqrt{D} - 2$.
    Moreover, when $D \modn{8} 1$ and $D > 9$, \Cref{eq:spin} shows that the spin of the prototypical surface $P\left(0,\frac{D - e^2}{4},1,e\right)$ is given by the parity of $\frac{e-f}{2}$, where $f$ is the conductor of $\O_D$. Thus, it is equal to $0$ if $e \modn{4} f$, and it is equal to $1$ if $e \nmodn{4} f$.
    
    We denote the prototypical surface associated with the reduced prototype $e$ by $\modelP_D(e)$, that is, $\modelP_D(e) = \modelP\left(0,\frac{D-e^2}{4},1,e\right) \in \Omega W_D$, as in the left of \Cref{fig:Le}.
    
    Let $R_D$ be the set of $e$-parameters of reduced prototypes of discriminant $D$ and, for $D$ satisfying $D \modn{8} 1$ and $D > 9$, let $R_D^{\varepsilon}$ be the set of $e$-parameters of reduced prototypes corresponding to surfaces in $\Omega W_D^{\varepsilon}$, for $\varepsilon \in \ZZ/2\ZZ$.
    That is, for any discriminant $D > 4$, let
    \[
        R_D = \{e \in \ZZ \ \mathbin{|}\  e \modn{2} D, e^2 < D, (e+2)^2 < D\},
    \]
    and, for $D \modn{8} 1$ and $D > 9$, let
    \[
        R_D^0 = \{e \in R_D \ \mathbin{|}\  e \modn{4} f\}
        \qquad \text{ and } \qquad
        R_D^1 = \{e \in R_D \ \mathbin{|}\  e \nmodn{4} f\},
    \]
    where $f$ is the conductor of $\O_D$, that is, the largest positive integer whose square divides $D$. Since, for fixed $D$, there is a one-to-one correspondence between reduced prototypes and their $e$-parameters, we will refer to $e \in R_D$ as a reduced prototype.
    
    \subsubsection{L-shaped representatives}
    We also use, in \Cref{sec:upper-bounds_H(2)}, an L-shaped version of the surface associated with a reduced prototype. This surface is obtained by the action of the diagonal matrix $\left(\begin{smallmatrix}\lambda^{-1} & 0 \\ 0 & 1 \end{smallmatrix}\right) \in \GL^+(2, \RR)$ on $\modelP_D(e)$, with $e \in R_D$, as in \Cref{fig:Le}. We denote such surface by $\modelL_D(e) \in \Omega W_D$.
    
    \begin{figure}[t!]
        \centering
        \includegraphics[scale=0.76]{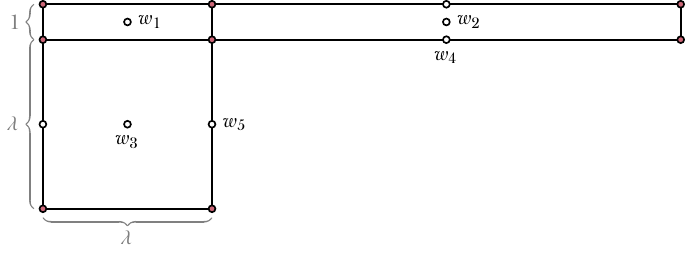}
        \raisebox{1.6cm}{$\xrightarrow{\;\left(\begin{smallmatrix} \lambda^{-1} & 0 \\ 0 & 1 \end{smallmatrix}\right)\;}$}
        \includegraphics[scale=0.76]{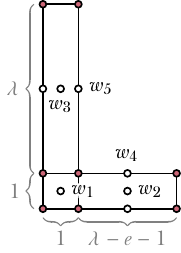}
        
        \caption{Prototypical surface $\modelP_D(e)$ (left) and L-shaped surface $\modelL_D(e)$ (right), for a reduced prototype $e \in R_D$. This example shows $\modelP_{73}(1)$ and $\modelL_{73}(1)$.}
        \label{fig:Le}
    \end{figure}

    \section{Action of parabolic elements}
    \label{sec:parabolic_H(2)}
 
    In this section, we show that the permutation group induced by the action of \emph{parabolic} elements of the Veech group contains a dihedral group.
    More precisely, using the reduced prototypes $\modelP_D(e)$, we compute a subgroup of $G_D(e) = G(\modelP_D(e))$, generated by the action of certain parabolic elements in the Veech group of $\modelP_D(e)$.

    \subsection{Parabolic elements and cylinder decompositions}\label{sec:parabolic-action}
    In what follows, we assume that a translation surface $(X, \omega)$ is endowed with the flat metric coming from $\omega$. If $(X, \omega)$ is viewed as a set of polygons in $\CC$ with side identifications, then a geodesic that does not contain any singularities in its interior is simply a straight line that respects said side identifications.
    
    A \emph{cylinder} of \emph{circumference} $w > 0$ and \emph{height} $h > 0$ on $X$ is a maximal open annulus filled by isotopic simple closed regular geodesics, isometric to $\RR/w\ZZ \times (0,h)$.
    Its \emph{modulus} is the ratio between its circumference $w$ and its height $h$, that is, $m = \frac{w}{h} > 0$.
    The \emph{core curve} of a cylinder is the simple closed geodesic identified with $\RR/w\ZZ \times \{h/2\}$.
    A \emph{saddle connection} is a geodesic joining two different singularities or a singularity to itself, with no singularities in its interior. Cylinders are bounded by parallel (concatenations of) saddle connections.
    
    A \emph{cylinder decomposition} of $X$ is a collection of parallel cylinders whose closures cover $X$. Given $\theta \in [0, 2\pi)$, we say that a cylinder decomposition is \emph{in direction $\theta$} if the angle between the horizontal direction and the core curve of any cylinder (or, equivalently, all cylinders) is $\theta$.
    
    By the work of Veech~\cite{Veech}, cylinder decompositions and parabolic elements are closely related.
    Indeed, fix an angle $\theta \in [0,2\pi)$ and let $R_\theta$ be the matrix rotating the plane (counterclockwise) by $\theta$. Given $t \in \RR$, let $T_t = \left(\begin{smallmatrix} 1 & t \\ 0 & 1 \end{smallmatrix}\right)$ be a matrix twisting the horizontal direction by $t$. Then, a matrix of the form $P_\theta^t = R_\theta T_t R_\theta^{-1}$ belongs to $\SL(X)$ if and only if there exists a cylinder decomposition $\{C_i\}_{i \in I}$ of $X$ in direction $\theta$ such that the moduli $m_i$ of the cylinders $C_i$ are pairwise rationally related, that is, if $r_{ij} = \frac{m_i}{m_j} \in \QQ$ for every $i,j \in I$.
    Furthermore, if $t \in \RR$ is an integer multiple of the modulus of each cylinder, that is, for each $i \in I$ there exist $k_i \in \ZZ$ such that $t = k_im_i$,
    then $P_\theta^t \in \SL(X)$. The corresponding affine transformation preserves the cylinder decomposition, as it twists each cylinder $C_i$ exactly $k_i$ times along itself.
    In particular, $P_\theta^t$ acts on each core curve $\gamma_i$ by a rotation by angle $k_i\pi$ (that is, $k_i$ half-turns). If the $k_i$ are minimal and positive, we call this action ``the'' \emph{twist} in direction $\theta$ and we say that $\theta$ is a \emph{parabolic direction}.
    
    In the case of Veech surfaces, parabolic directions are easily characterised by the so-called \emph{Veech dichotomy}~\cite{Veech}. Indeed, a direction $\theta \in [0, 2\pi)$ is parabolic if and only if there exists a saddle connection or a closed geodesic in direction $\theta$, meaning that the angle between the horizontal and the curve is $\theta$.
    
    \subsection{Cylinder decompositions in \texorpdfstring{$\H(2)$}{H(2)}}\label{sec:cylinder-moduli}
    Let $X \in \H(2)$. It is known that there are at most two cylinders in a cylinder decomposition of $X$. Moreover, the core curve of each cylinder contains exactly two Weierstrass points \cite[Remark~5]{Pardo:non-varying}.
    In particular, a parabolic element $P_\theta^t \in \SL(X)$ exchanges two (different) regular Weierstrass points $x,y \in X$ if and only if 
    \begin{enumerate}
        \item $X$ admits a cylinder decomposition $\{C_i\}_i$ in direction $\theta$;
        \item $x$ and $y$ lie on the same core curve $\gamma_i$ of one of the cylinders $C_i$; and
        \item the corresponding integer $k_i$ of half-turns along $\gamma_i$ is odd (recall that $t = k_i m_i$).
    \end{enumerate}
    
    Let $\theta \in [0, 2\pi)$ and assume that $\{C_1,C_2\}$ is a two-cylinder decomposition of $X$ in direction $\theta$. We can write
    \[
        r = \frac{m_1}{m_2} = \frac{p}{q} \in \QQ,
    \]
    where $p$ and $q$ are positive integers with $\gcd(p,q) = 1$. Then, the twist in direction $\theta$ is described by taking $k_1 = q$ and $k_2 = p$. In other words, if $t = qm_1 = pm_2$, then the parabolic element $P_\theta^t$ belongs to $\SL(X)$, and $t > 0$ is minimal for this property. Since $t$ will be always chosen in this way, we can omit it from the notation and define $r_\theta = \frac{m_1}{m_2} = \frac{p}{q}$.
    
    By the previous discussion, we can compute the action of $P_\theta^t$ at the level of (regular) Weierstrass points, which we denote by $\tau_\theta \in \Sym(\{w_1,\dots,w_5\})$.
    Indeed, $\tau_\theta$ exchanges the two regular Weierstrass points in $\gamma_i$, the core curve of $C_i$, if and only if $k_i$ is odd. Since $\gcd(k_1,k_2) = \gcd(q,p) = 1$, $\tau_\theta$ is either a transposition or a product of two disjoint transpositions.
    
    \noindent\textbf{Notation.}
    In what follows we denote
    \begin{align*}
        \frac{\text{odd}}{\text{odd}} &= \left\{\frac{p}{q} \in \QQ \ \mathbin{\Big|}\  p, q \in \NN, \gcd(p,q) = 1, p \modn{2} q \modn{2} 1\right\}, \\
        \frac{\text{even}}{\text{odd}} &= \left\{\frac{p}{q} \in \QQ \ \mathbin{\Big|}\  p, q \in \NN, \gcd(p,q) = 1, p \modn{2} 0, q \modn{2} 1\right\} \qquad \text{and} \\
        \frac{\text{odd}}{\text{even}} &= \left\{\frac{p}{q} \in \QQ \ \mathbin{\Big|}\  p, q \in \NN, \gcd(p,q) = 1, p \modn{2} 1, q \modn{2} 0\right\}.
    \end{align*}
    
    Moreover, for a two-cylinder decomposition of a surface in $\H(2)$, we always choose $C_1$ to be the cylinder with the largest circumference. We will refer to $C_1$ as the \emph{long cylinder} and, accordingly, to $C_2$ as the \emph{short cylinder}.
    
    We can summarise the previous discussions in the form of the following lemma which we will repeatedly use throughout the article.
    
    \begin{lem}
        \label{lem:summary_twists}
        Let $X$ be a translation surface in $\H(2)$. Let $\theta \in [0, 2\pi)$ be a parabolic direction of $X$ and assume that there exists a two-cylinder decomposition $\{C_1, C_2\}$ of $X$ in direction $\theta \in [0, 2\pi)$. Recall that $r_\theta = \frac{m_1}{m_2}$, where $m_i$ is the modulus of the cylinder $C_i$, and that $\tau_\theta \in \Sym(\{w_1, \dotsc, w_5\})$ is the action of the twist in direction $\theta$ on the regular Weierstrass points $w_1, \dotsc, w_5$ of $X$. We have that
        \begin{enumerate}
            \item if $r_\theta \in \frac{\mathrm{odd}}{\mathrm{odd}}$, then $\tau_\theta$ is a product of two disjoint transpositions, exchanging the regular Weierstrass points in the core curves of $C_1$ and $C_2$;
            \item if $r_\theta \in \frac{\mathrm{even}}{\mathrm{odd}}$, then $\tau_\theta$ only exchanges the two regular Weierstrass points in the core curve of the long cylinder $C_1$; and
            \item if $r_\theta \in \frac{\mathrm{odd}}{\mathrm{even}}$, then $\tau_\theta$ only exchanges the two regular Weierstrass points in the core curve of the short cylinder $C_2$.
            \qed
        \end{enumerate}
    \end{lem}

\subsection{Horizontal and vertical twists on reduced prototypes} \label{sec:horizontal_vertical_twists_reduced_prototypes}
    Let $D > 0$ be a real quadratic discriminant. Given a splitting prototype $(a,b,c,e)$ of discriminant $D$, the horizontal direction induces a two-cylinder decomposition $\{C_1, C_2\}$ of $\modelP(a,b,c,e)$, where $C_1$ is the long cylinder. Referring to \Cref{fig:prototypical} we see that the moduli $m_i$ of the cylinders $C_i$ are given by
    \[
        m_1 = \frac{b}{c}
        \qquad \text{ and } \qquad
        m_2 = \frac{\lambda}{\lambda} = 1.
    \]
    The ratio of the moduli is then
    \[
        r_{\mathrm{h}} = r_0 = \frac{m_1}{m_2} = \frac{b}{c} = \frac{D-e^2}{4c^2},
    \]
    where the label ``$\mathrm{h}$'' stands for \emph{horizontal}.
    
    Consider now the case of a reduced prototype $e \in R_D$. In this case, we have that $r_{\mathrm{h}} = b = \frac{D - e^2}{4}$. Moreover, the vertical direction induces a two-cylinder decomposition $\{C_1, C_2\}$ of $\modelP_D(e)$, where $C_1$ is the long cylinder. From \Cref{fig:Le} (left) we see that the corresponding moduli are
    \[
        m_1 = \frac{\lambda + 1}{\lambda}
        \qquad \text{ and } \qquad
        m_2 = \frac{1}{b - \lambda} = \frac{1}{\frac{D - e^2}{4} - \lambda}.
    \]
    Since $\lambda = \frac{e + \sqrt{D}}{2}$, it follows from an elementary computation that the ratio of the moduli is
    \[
        r_{\mathrm{v}} = r_{\pi/2} = \frac{m_1}{m_2} = \frac{D - (e + 2)^2}{4} = b - e - 1,
    \]
    where the label ``$\mathrm{v}$'' stands for \emph{vertical}. In particular, we obtain that $r_{\mathrm{v}} = r_{\mathrm{h}} - e - 1$.
    
    We can now summarise this discussion for reduced prototypes:
    \begin{lem}
        \label{lem:summary_twists_reduced_prototypes}
        Let $D > 4$ be a real quadratic discriminant. Let $e \in R_D$ be a reduced prototype and let $b = \frac{D - e^2}{4} \in \NN$. Let $X = \modelP_D(e)$ be the prototypical surface associated with the reduced prototype $(0, b, 1, e)$ and let $w_1, \dotsc, w_5$ be its regular Weierstrass points. We set $\tau_{\mathrm{h}} = \tau_0$ and $\tau_{\mathrm{v}} = \tau_{\pi/2}$, that is, $\tau_{\mathrm{h}}$ and $\tau_{\mathrm{v}}$ are the permutations in $\Sym(\{w_1, \dotsc, w_5\})$ associated with the horizontal and vertical twists, respectively. Then, we have that the permutations $\tau_{\mathrm{h}}$ and $\tau_{\mathrm{v}}$ always exchange the regular Weierstrass points in the long horizontal and long vertical cylinders, respectively. Moreover, $\tau_{\mathrm{h}}$ exchanges the regular Weierstrass points in the short horizontal cylinder if and only if $b$ is odd. Similarly, $\tau_{\mathrm{v}}$ exchanges the regular Weierstrass points in the short vertical cylinder if and only if $b - e - 1$ is odd.
    \end{lem}
    
    \begin{proof}
        From the previous discussion, we have that $r_{\mathrm{h}} = b$ and $r_{\mathrm{v}} = b - e - 1$. Thus, $r_{\mathrm{h}}$ and $r_{\mathrm{v}}$ are integers. In particular, they belong to $\frac{\mathrm{odd}}{\mathrm{odd}} \sqcup \frac{\mathrm{even}}{\mathrm{odd}}$. By \Cref{lem:summary_twists}, the horizontal and vertical twists always exchange the regular Weierstrass points in the corresponding long cylinders. Finally, \Cref{lem:summary_twists} also implies that whether they exchange the regular Weierstrass points in the corresponding short cylinders only depends on the parities of $b$ and $b - e - 1$.
    \end{proof}
    
    \begin{rem}
    \label{rem:butterfly_H(2)}
        We have refrained from introducing the notion of \emph{butterfly moves} to simplify the exposition.
        In fact, several cylinder decompositions on a prototypical surface can be thought of as a butterfly move on splitting prototypes (see McMullen's work \cite[Section~7]{McMullen:discriminant_spin}).
        For the case of reduced prototypes, the vertical cylinder decomposition corresponds to the butterfly move $B_\infty$, which maps the prototypical surface $\modelP_D(e)$ into the prototypical surface $\modelP_D(-e - 4)$ \cite[Theorem~ 7.5]{McMullen:discriminant_spin}. Thus, the fact that the action induced by the vertical twist depends on the parity of $\frac{D - (e + 4)^2}{4}$ can also be explained in this way.
    \end{rem}
    
    In what follows, given a prototypical surface, we fix $w_1,\dots,w_5$ as in \Cref{fig:prototypical} and identify $\Sym_5$ with $\Sym(\{w_1,\dots,w_5\})$ in the natural way. Thus, for example, the transposition $(i\;j) \in \Sym_5$ corresponds to $(w_i\;w_j) \in \Sym(\{w_1,\dots,w_5\})$.
    
    \begin{prop}\label{prop:D=0 mod8}\label{prop:D=4 mod8}\label{prop:D=0 mod4}
        Let $D > 4$ be a real quadratic discriminant with $D \modn{4} 0$ and $X \in \Omega W_D$. Then, $G(X)$ has a subgroup conjugate to $\Dih_4$ generated by the action of two parabolic elements in $\SL(X)$.
    \end{prop}
    
    \begin{proof}
    Let $e \in R_D$. Since $W_D$ is irreducible, there is an element in $\GL^+(2,\RR)$ that maps $X$ to $\modelP_D(e)$. In particular $G(X)$ is conjugate to $G_D(e) = G(\modelP_D(e))$, and the parabolic elements of $\SL(X)$ are conjugate to the parabolic elements of $\SL(\modelP_D(e))$. Thus, it is enough to prove the result for $G_D(e)$.
    
    By definition of $R_D$, we have that $e \modn{2} D$, so $e$ is even. Therefore, $r_{\mathrm{h}} \nmodn{2} r_{\mathrm{h}} - e - 1 = r_{\mathrm{v}}$.
    
    In the case that $D \modn{8} e^2$, $r_{\mathrm{h}} = b = \frac{D - e^2}{4}$ is even and $r_{\mathrm{v}}$ is odd. Thus, by \Cref{lem:summary_twists_reduced_prototypes}, the horizontal twist induces the permutation $\tau_{\mathrm{h}} = (1\;2)$, and the vertical twist induces the permutation $\tau_{\mathrm{v}} = (1\;3)(2\;4)$.
    It follows that
    \[
        \langle \tau_{\mathrm{h}}, \tau_{\mathrm{v}}\rangle = \DihFour{1}{3}{2}{4} \leqslant G_D(e).
    \]
    
    On the other hand, if $D \nmodn{8} e^2$, $r_{\mathrm{h}} = b = \frac{D - e^2}{4}$ is odd and $r_{\mathrm{v}}$ is even. Thus, by \Cref{lem:summary_twists_reduced_prototypes}, the horizontal twist induces the permutation $\tau_{\mathrm{h}} = (1\;2)(3\;5)$, and the vertical twist induces the permutation $\tau_{\mathrm{v}} = (1\;3)$.
    It follows that
    \[
        \langle \tau_{\mathrm{h}}, \tau_{\mathrm{v}}\rangle = \DihFour{1}{2}{3}{5} \leqslant G_D(e).
    \]
    \end{proof}
    
    \begin{prop}\label{prop:D=5 mod8}
        Let $D > 4$ be a real quadratic discriminant with $D \modn{8} 5$ and $X \in \Omega W_D$. Then, $G(X)$ has a subgroup conjugate to $\Dih_5$ generated by the action of two parabolic elements in $\SL(X)$.
    \end{prop}
    \begin{proof}
    Let $e \in R_D$. Since $W_D$ is irreducible, there is an element in $\GL^+(2,\RR)$ that maps $X$ to $\modelP_D(e)$. In particular $G(X)$ is conjugate to $G_D(e) = G(\modelP_D(e))$, and the parabolic elements of $\SL(X)$ are conjugate to the parabolic elements of $\SL(\modelP_D(e))$. Thus, it is enough to prove the result for $G_D(e)$.
    
    By definition of $R_D$, we have that $e \modn{2} D$, so $e$ is odd. Therefore, $r_{\mathrm{h}} \modn{2} r_{\mathrm{h}} - e - 1 = r_{\mathrm{v}}$ and $e^2 \modn{8} 1$. Hence, $r_{\mathrm{h}} = b = \frac{D - e^2}{4}$ is odd.
    Thus, by \Cref{lem:summary_twists_reduced_prototypes}, the horizontal twist induces the permutation $\tau_{\mathrm{h}} = (1\;2)(3\;5)$, and the vertical twist induces the permutation $\tau_{\mathrm{v}} = (1\;3)(2\;4)$.
    It follows that
    \[
        \langle \tau_{\mathrm{h}}, \tau_{\mathrm{v}}\rangle = \DihFive{1}{2}{5}{4}{3} \leqslant G_D(e).
        \qedhere
    \]
    \end{proof}

    \begin{prop}\label{prop:D=1 mod8}
        Let $D > 4$ be a real quadratic discriminant with $D \modn{8} 1$ and $X \in \Omega W_D$. Then, $G(X)$ has a subgroup conjugate to $\Sym_3$ generated by the action of two parabolic elements in $\SL(X)$.
    \end{prop}
    \begin{proof}
    In this case $W_D = W_D^0 \sqcup W_D^1$.
    Let $\varepsilon \in \ZZ/2\ZZ$ be the corresponding spin, that is, such that $X \in \Omega W_D^\varepsilon$.
    Let $e \in R_D^\varepsilon$. Since $W_D^\varepsilon$ is irreducible, there is an element in $\GL^+(2,\RR)$ that maps $X$ to $\modelP_D(e)$. In particular $G(X)$ is conjugate to $G_D(e) = G(\modelP_D(e))$, and, similarly, the parabolic elements of $\SL(X)$ are conjugate to the parabolic elements of $\SL(\modelP_D(e))$. Thus, it is enough to prove the result for $G_D(e)$.
    
    By definition of $R_D$, we have that $e \modn{2} D$, so $e$ is odd. Therefore, $r_{\mathrm{h}} \modn{2} r_{\mathrm{h}} - e - 1 = r_{\mathrm{v}}$ and $e^2 \modn{8} 1$. Hence, $r_{\mathrm{h}} = b = \frac{D - e^2}{4}$ is even. Thus, by \Cref{lem:summary_twists_reduced_prototypes}, the horizontal twist induces the permutation $\tau_{\mathrm{h}} = (1\;2)$, and the vertical twist induces the permutation $\tau_{\mathrm{v}} = (1\;3)$, regardless of the spin.
    It follows that
    \[
        \langle \tau_{\mathrm{h}}, \tau_{\mathrm{v}}\rangle = \Sym(\{1, 2, 3\}) \leqslant G_D(e).
        \qedhere
    \]
    \end{proof}
    
    \begin{prop}\label{prop:proper}
        Let $D > 4$ be a real quadratic discriminant with $D \modn{8} 1$ and $D \notin \{9, 33\}$. Let $X \in \Omega W_D$. Then, $\SL(X)$ contains a parabolic element that induces a product of two disjoint transpositions inside $G(X)$.
    \end{prop}
    \begin{proof}
        Let $D$ be as in the statement of the proposition. Observe that, since $D \neq 9$ and $D \modn{8} 1$, $W_D$ is reducible and $W_D = W_D^0 \sqcup W_D^1$. It is enough to find a prototypical surface for each spin such that its horizontal twist induces a product of two disjoint transpositions.
        
        Consider the choices $b = 4k + 2$, $c = 2$ and $e \in \{-1, -3, -5, -7\}$. Observe that, with these choices, we have that
        \[
            D = e^2 + 4bc = e^2 + 16 + 32k = \begin{cases}
                17 + 32k & \text{if } e = -1 \\
                25 + 32k & \text{if } e = -3 \\
                41 + 32k & \text{if } e = -5 \\
                65 + 32k & \text{if } e = -7,
            \end{cases}
        \]
        where $k \geq 0$.
        Then, both $(0,b,c,e)$ and $(1,b,c,e)$ are splitting prototypes. Indeed, $a < \gcd(b,c) = 2$ if $a \in \{0, 1\}$, $c + e \leq 1 < 2 \leq b$, and $\gcd(a, b, c, e) = 1$ since $c = 2$ and $e$ is odd. Moreover, since $b$ and $c$ are even, \Cref{eq:spin} shows that these two prototypes correspond to different spins.
        
        Finally, since $r_0 = \frac{b}{c} = 2k + 1$, we have $r_0 \in \frac{\text{odd}}{\text{odd}}$. Thus, by \Cref{lem:summary_twists}, $G(X)$ contains an element that is the product of two disjoint transpositions.
    \end{proof}
    
    \begin{rem}
        \label{rem:D=9/33}
        If $D \in \{9, 33\}$, an exhaustive search shows that no prototype satisfies that the ratio of moduli $r_0$ belongs to $\frac{\text{odd}}{\text{odd}}$ (for both spins in the case $D= 33$).
        In particular, the proof of \Cref{prop:proper} does not work in these cases. Since any two parabolic directions can be mapped to the horizontal and vertical directions on some prototypical surface, we conclude that, if we only restrict to only two parabolic elements, the largest group that can be generated by the action of these elements on the regular Weierstrass points of a surface in $\Omega W_D$ is $\Sym_3 \cong \Dih_3$. However, it is still possible to exhibit a \emph{third} explicit parabolic direction on particular representatives that allows us to conclude that $G(X) \cong \Sym_3 \times \Sym_2 \cong \Dih_6$. See~\Cref{sec:D=9,sec:D=33} below.
    \end{rem}

\section{Restrictions for permutations}
\label{sec:upper-bounds_H(2)}
    In this section, we constrain the group of permutations of Weierstrass points. To this end, we use geometric restrictions on the location of the Weierstrass points, and algebraic restrictions on the coefficients of the elements of the Veech group.
    
    Since these restrictions differ in the arithmetic case (when $D$ is a perfect square) and the non-arithmetic case, we separate the analysis in different subsections. We start with the arithmetic case, which is more straightforward.

\subsection{Arithmetic case}
\label{sec:arithmetic_H(2)}

    In this section, we analyse the case where $D \geq 9$ is a perfect square, say $D = d^2$ for $d \geq 3$ an integer. Observe first that if $D$ is odd (or, equivalently, if $d$ is odd), then $D \modn{8} 1$. Thus, the case $D \modn{8} 5$ does not arise.
    
    The Weierstrass locus $\Omega W_{d^2}$ consists of $\GL^+(2,\RR)$-orbits of square-tiled surfaces, which we will properly define in what follows. A (primitive) \emph{square-tiled surface} is a translation surface $(X, \omega)$ such that the set of periods of $\omega$ is the lattice $\ZZ[i]$ in $\CC$. That is, the integration of $\omega$ along any saddle connection produces an element of $\ZZ[i]$. For such a surface, we obtain again by integrating $\omega$ a holomorphic map
    \[
        p \colon X \to E = \CC/\ZZ[i],
    \] 
    which can be normalised so it is branched only over $z = 0$.
    
    If we assume that $(X, \omega) \in \H(2)$ is a square-tiled surface, we have, by McMullen's work \cite[Section~6]{McMullen:discriminant_spin}, that $(X,\omega) \in \Omega W_{d^2}$ for some $d$.
    Moreover, the degree of the covering map $p$ is equal to $d$, and the $d$ preimages of the unit-square fundamental domain for $E$ provide the square tiling of $X$.
    
    Since Weierstrass points are exactly those points invariant for the hyperelliptic involution, a Weierstrass point $w \in X$ can only project down to a $2$-torsion point in $E$.
    Clearly, modulo the lattice $\ZZ[i]$, the $2$-torsion points in $E$ are $0$, $h = \frac{1}{2}$,  $v = \frac{i}{2}$ and $c = \frac{1}{2} + \frac{i}{2}$.
    We say that the Weierstrass point $w \in X$ is of type $0$, $h$, $v$ or $c$, accordingly.
    In addition, a Weierstrass point $w \in X$ is called \emph{integral} if $p(w) = 0 \bmod\ZZ[i]$, that is, if $w$ is of type~$0$.
    
    To conclude the proof of \Cref{thm:main_H(2)} for the arithmetic case, we use these geometric constraints imposed by the projection of the regular Weierstrass points of $X$. In fact, if an affine transformation on $X$ maps a point of a certain type $\mathrm{t}_0$ to another point of type $\mathrm{t}_1$, then it has to map \emph{every} point of type $\mathrm{t}_0$ to a point of type $\mathrm{t}_1$, since it induces a well-defined quotient map on $E$. Moreover, the set of integral points must be preserved.
    
    Recall that the HLK-invariant is defined as the number of integral regular Weierstrass points, together with the (unordered) list of number of regular Weierstrass points of type $h$, $v$ and $c$. We will use square brackets to denote unordered lists. The restrictions of the previous paragraph are precisely equivalent to the fact that the HLK-invariant is preserved.
    
    For $e \in R_D$, it is not difficult to see that the L-shaped representatives $\modelL_D(e)$ in \Cref{fig:Le} are in fact square-tiled surfaces when $D = d^2$ is a perfect square. Indeed, we have that $\lambda = \frac{e + d}{2} \in \NN$, so these surfaces consist on $\lambda+1$ unit squares stacked upwards, and $\lambda-e$ unit squares stacked rightwards.    See \Cref{fig:rep-even,fig:rep-odd}.
    In particular, the type of $w_1,\dots,w_5 \in \modelL_D(e)$ depends only on the parity of $\lambda$ and $\lambda-e-1$.
    More precisely, we have that
    \begin{itemize}
        \item $w_1$ is always of type $c$;
        \item $w_2$ is of type $c$ whenever $\lambda-e-1$ is odd, and of type $v$ otherwise;
        \item $w_3$ is of type $c$ whenever $\lambda$ is odd, and of type $h$ otherwise;
        \item $w_4$ is of type $h$ whenever $\lambda-e-1$ is odd, and integral otherwise; and
        \item $w_5$ is of type $v$ whenever $\lambda$ is odd, and integral otherwise.
    \end{itemize}
    Furthermore, recall that these L-shaped surfaces are essentially prototypical surfaces for a reduced prototype (up to a scaling). This allows us to obtain the following simple lemma:
    
    \begin{lem} \label{lem:spin_square_tiled}
        Let $d > 2$ be an odd integer and let $e \in R_D$. Set $\lambda = \frac{e + d}{2}$. Then, $\modelL_{d^2}(e) \in \Omega W_{d^2}^0$ if $\lambda$ is odd, and $\modelL_{d^2}(e) \in \Omega W_{d^2}^1$ if $\lambda$ is even.
    \end{lem}
    
    \begin{proof}
        Fix an odd integer $d > 2$ and let $e \in R_{d^2}$. Recall that $\modelL_{d^2}(e)$ is a scaling of the reduced prototype $\left(0, \frac{d^2 - e^2}{4}, 1, e\right)$. Thus, its spin coincides with the spin of the prototypical surface $P\left(0, \frac{d^2 - e^2}{4}, 1, e\right)$.
        
        Since $D = d^2$, the conductor $f$ of $D$ is equal to $d$. By \Cref{eq:spin}, we see that the spin of this prototypical surface is equal to the parity of $\frac{e - d}{2}$. Since $d$ is odd, the spin is also equal to the parity of $\lambda + 1 = \frac{e + d}{2} + 1$, which gives the desired conclusion.
        
        This lemma also follows directly from McMullen's work \cite[Theorem~6.1]{McMullen:discriminant_spin} by counting the number of integral points in each case.
    \end{proof}
    
    \begin{prop}\label{prop:arith-d2=0 mod4}
        Let $d > 2$ be an even integer and let $X \in \Omega W_{d^2}$. Then, $G(X)$ is contained in a group conjugate to $\Dih_4$.
    \end{prop}
    \begin{proof}
        Fix an even integer $d > 2$. Let $e \in R_D$. When $d^2 \modn{8} e^2$, we have that $\lambda$ is even and $\lambda-e-1$ is odd.
        Thus, $w_1,w_2$ are of type $c$, $w_3,w_4$ are of type $h$ and $w_5$ is integral. Thus, the HLK-invariant is $(1, [2, 2, 0])$. See \Cref{fig:rep-even}~(left).
        It follows that
        \[
            G_{d^2}(e) \leqslant \DihFour{1}{3}{2}{4}.
        \]
        
        When $d^2 \nmodn{8} e^2$, we have that $\lambda$ is even and $\lambda-e-1$ is odd.
        Thus, $w_1,w_3$ are of type $c$, $w_2,w_5$ are of type $v$ and $w_4$ is integral. Thus, the HLK-invariant is $(1, [2, 2, 0])$.
        See \Cref{fig:rep-even} (right).
        It follows that
        \[
            G_{d^2}(e) \leqslant \DihFour{1}{2}{3}{5}.
            \qedhere
        \]
    \end{proof}
    \begin{figure}[t!]
        \centering
        \includegraphics[scale=0.8]{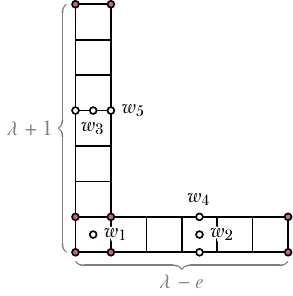}
        \qquad
        \includegraphics[scale=0.8]{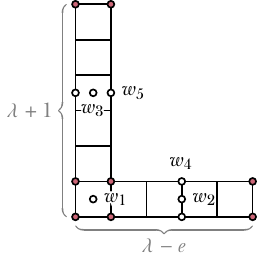}
        \caption{The L-shaped square-tiled surface $\modelL_{d^2}(e) \in \Omega W_{d^2}$, when $d$ is even.
        The case $d^2 \modn{8} e^2$ (left) and $d^2 \nmodn{8} e^2$ (right). These examples show $\modelL_{12^2}(0)$ and $\modelL_{10^2}(0)$, respectively.}
        \label{fig:rep-even}
    \end{figure}
    
    \begin{prop}\label{prop:arith-d2=1 mod8}
        Let $d > 2$ be an odd integer and let $X \in \Omega W_{d^2}$. Then, $G(X)$ is contained in a group isomorphic to $\Sym_3\times\Sym_2 \cong \Dih_6$.
    \end{prop}
    \begin{proof}
        Fix an odd integer $d > 2$. Let $e \in R_D$. When $\lambda$ is odd, \Cref{lem:spin_square_tiled} shows that $\modelL_{d^2}(e) \in \Omega W_{d^2}^0$. In this case, we have no integral regular Weierstrass points, as in \Cref{fig:rep-odd}~(left).
        Moreover, $\lambda-e-1$ is even and therefore, $w_1,w_2,w_3$ are of type $c$, $w_4$ is of type $h$ and $w_5$, of type $v$. Thus, the HLK-invariant is $(0, [3, 1, 1])$.
        It follows that
        \[
            G_{d^2}(e) \leqslant \Sym(\{1,2,3\}) \times \Sym(\{4,5\}).
        \]

        Observe that the previous situation is the only one that can arise for the particular case of $d = 3$, as $R_{3^2} = \{-1\}$ (and, hence, $\lambda = 1$).
        
        When $\lambda$ is even, \Cref{lem:spin_square_tiled} shows that $\modelL_{d^2}(e) \in \Omega W_{d^2}^1$. In this case, we have two integral regular Weierstrass points, namely $w_4$ and $w_5$, as in \Cref{fig:rep-odd}~(right). Moreover, $\lambda-e-1$ is odd and therefore, $w_1$, $w_2$ and $w_3$ are of type $c$, $v$ and $h$, respectively. Thus, the HLK-invariant is $(2, [1, 1, 1])$.
        It follows that
        \[
            G_{d^2}(e) \leqslant \Sym(\{1,2,3\}) \times \Sym(\{4,5\}).
            \qedhere
        \]
    \end{proof}
    \begin{figure}[t!]
        \centering
        \includegraphics[scale=0.8]{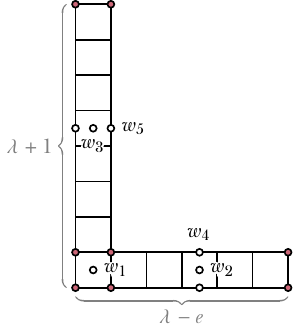}
        \qquad
        \includegraphics[scale=0.8]{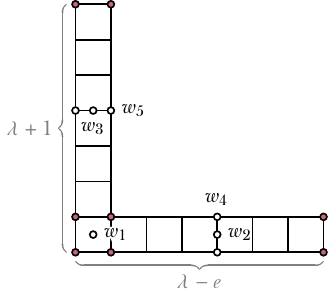}
        \caption{The L-shaped square-tiled surface $\modelL_{d^2}(e) \in \Omega W_{d^2}$, when $d$ is odd.
        The case $W_{d^2}^0$ (left) and $W_{d^2}^1$ (right). These examples show $\modelL_{13^2}(1)$ and $\modelL_{13^2}(-1)$, respectively.}
        \label{fig:rep-odd}
    \end{figure}

\subsection{Non-arithmetic case}
\label{sec:non-arithmetic_H(2)}
    In this section, we analyse the case where $D > 9$ is a real quadratic discriminant that is not a perfect square.
    We first show that if $D$ is a quadratic residue modulo~$8$, then the problem can be treated completely analogously to the arithmetic case.
    In the case where $D$ is a quadratic nonresidue modulo~$8$, that is, when $D \modn{8} 5$, we show that the action of the affine group on the set of regular Weierstrass points can be restricting by using similar ideas, although it cannot be reduced to the arithmetic case.
    
    \begin{lem} \label{lem:SL_O_D_H(2)}
        Let $D > 4$ be a real quadratic discriminantl. Let $e \in R_D$ be any reduced prototype. Then, $\SL(\modelL_D(e)) \leqslant \SL(2, \O_D)$.
    \end{lem}
    \begin{proof}
        Fix $A \in \SL(\modelL_D(e))$. We have that $A$ must map every saddle connection into a saddle connection, and, therefore, must preserve the set of holonomy vectors of saddle connections.
        
        In particular, consider the two saddle connections obtained by tracing a ray from from the lower-left corner of \Cref{fig:Le} (right) and moving rightwards or upwards until the singularity is again met. The holonomy vectors of these saddle connections are $(1, 0)^\tr$ and $(0, 1)^\tr$, respectively. We conclude that both columns of $A$ are holonomy vectors of saddle connections.
        
        Finally, observe that the holonomy vector of any saddle connection must belong to $\O_D \times \O_D$, since every side length and translation vector of $\modelL_D(e)$ belongs to $\O_D$.
    \end{proof}
    
    In what follows, we are interested in the relative position of the regular Weierstrass points in $\modelL_D(e)$, with respect to the singularity. In particular, we restrict our attention to elements in $\frac{1}{2}(\O_D \times \O_D)$.
    
    Fix $\rho \in \QQ(\sqrt{D})$ such that $\O_D = \ZZ[\rho]$; the choice of $\rho$ will be made explicit later as needed.
    Each element $x$ of the field $\QQ(\sqrt{D})$ can be written uniquely as $x = p + q\rho$ with $p, q \in \QQ$. We say that $p$ is the \emph{rational part} of $x$ (note that this depends on the choice of $\rho$). Moreover, we define $\fr(x)$ as the fractional part of the rational part of $x$. That is, $\fr(x) = \{p\} = (p\bmod{1})$.
    
    \noindent\textbf{Notation.}
    \begin{itemize}[beginpenalty=10000,endpenalty=10000]
        \item For $A = \bigl(\begin{smallmatrix} \alpha & \beta \\ \gamma & \delta \end{smallmatrix}\bigr) \in \SL(2, \O_D)$, we write $\eta = p_\eta + q_\eta\rho \in \O_D$, where $p_\eta, q_\eta \in \ZZ$, for each $\eta \in \{\alpha, \beta, \gamma, \delta\}$. We also write $P_A = \bigl(\begin{smallmatrix} p_\alpha & p_\beta \\ p_\gamma & p_\delta \end{smallmatrix}\bigr)$ and $Q_A = \bigl(\begin{smallmatrix} q_\alpha & q_\beta \\ q_\gamma & q_\delta \end{smallmatrix}\bigr)$, so $A = P_A + Q_A \rho$.
        \item For $v \in \frac{1}{2}(\O_D \times \O_D)$, we write $v = \frac{1}{2}(p_1 + q_1\rho, p_2 + q_2\rho)$, where $p_i, q_i \in \ZZ$ for each $i \in \{1,2\}$. We also write $p_v = (p_1, p_2)$ and $q_v = (q_1, q_2)$, so $v = \frac{1}{2}(p_v + q_v \rho)$.
    \end{itemize}
    Then, 
    \begin{equation}
        A v = \frac{1}{2}(P_A + Q_A \rho)(p_v + q_v \rho) = \frac{1}{2} (P_A p_v + Q_A q_v \rho^2 + (P_A q_v + Q_A q_v) \rho) \label{eq:Av}
    \end{equation}
    
    \subsubsection{Quadratic residues modulo~$8$}
    \label{sec:quadratic-residues_H(2)} We will now show that, if the discriminant is a quadratic residue modulo $8$, then the constraints on the group of permutations on the regular Weierstrass points are exactly the same as in the arithmetic case. We do this by showing that if we only keep track of rational parts, then there is an analogue of the HKL-invariant that has to be preserved. The next lemma provides the first step in this direction:
    
    \begin{lem}
    Let $D > 0$ be a real quadratic discriminant and assume that it is a quadratic residue modulo $8$, say $D \modn{8} d^2$ for some $d \in \ZZ$, and let $A \in \SL(2, \O_D)$. Take $\rho = \frac{\sqrt{D} - d}{2}$. Then, we have $\fr(A v) = \frac{1}{2} P_A p_v \bmod{1}$ for any $v \in \frac{1}{2}(\O_D \times \O_D)$. 
    Moreover, $P_A$ is non-singular modulo~$2$.
    \end{lem}
    
    \begin{proof}
        We have that
        \[
            \rho^2
            = \frac{D + d^2 - 2d\sqrt{D}}{4}
            = \frac{D - d^2}{4} - d\rho
        \]
        By \Cref{eq:Av}, we obtain that the rational part of $A v$ is
        \[
            \frac{1}{2} P_A p_v + \frac{D - d^2}{8} Q_A q_v
        \]
        Since $\frac{D - d^2}{8}$ is an integer, this implies $\fr(A v) = \frac{1}{2} P_A p_v \bmod 1$.
        
        Now, a direct computation shows that
        \[
            1 = \det(A) = \det(P_A + Q_A \rho) = \det(P_A) + \det(Q_A) \rho^2 + \begin{vmatrix}
                p_\alpha & p_\beta \\ q_\gamma & q_\delta
            \end{vmatrix} \rho + \begin{vmatrix}
                q_\alpha & q_\beta \\ p_\gamma & p_\delta
            \end{vmatrix} \rho.
        \]
        Taking rational parts, we obtain $\det(P_A) + \frac{D - d^2}{4} \det(Q_A) = 1$ and, therefore, $\det(P_A) \modn{2} 1$.
    \end{proof}
    
    The previous lemmas show that the cases $D \modn{8} d^2$ can be treated as the arithmetic case.
    That is, if we set $\rho = \frac{\sqrt{D} - d}{2}$ and we only keep track of the fractional part of the rational part of the vectors $v_i \in \frac{1}{2}(\O_D \times \O_D)$ associated with the regular Weierstrass points $w_i$, for $i \in \{1, \dots, 5\}$, on a surface $X \in \H(2)$, then the action of $\SL(X)$ behaves as a linear action on points over the $2$-torsion points on a torus. This is exactly the case for the regular Weierstrass points in the arithmetic case (see \Cref{sec:arithmetic_H(2)}). As in that case, we say that a regular Weierstrass point $w_i$ is of type $0$, $h$, $v$ or $c$ depending on where $\fr(v_i)$ lies on the unit square. Since $P_A$ is non-singular modulo 2, we conclude that it acts bijectively on the sets of $1$-torsion and $2$-torsion points of the unit torus, so we may extend the definition of the HLK-invariant to this case. More precisely, we define the HLK-invariant as the following data:
    \begin{itemize}
        \item the number of points of type $0$; and
        \item the (unordered) list of number of points of type $h$, $v$ and $c$.
    \end{itemize}
    
    Given $d \in \ZZ$, a real quadratic discriminant $D > 4$ with $D \modn{8} d^2$ and $e \in R_D$, it is straightforward to compute $\fr(v_i)$ for each $i \in \{1, \dotsc, 5\}$ in the surface $\modelL_D(e)$. The following lists summarise these computations.
    \begin{enumerate}
        \item If $D \modn{8} 0$ or $D \modn{8} 4$, we have two cases:
        \begin{enumerate}
            \item If $e^2 \modn{8} D$, the HLK-invariant is $(1, [2, 2, 0])$ since
            \begin{gather*}
    	        \fr(v_1) = (1/2, 1/2), \quad
    	        \fr(v_2) = (1/2, 1/2), \quad
            	\fr(v_3) = (1/2, 0), \\
            	\fr(v_4) = (1/2, 0)
            	\quad \text{ and } \quad
            	\fr(v_5) = (0, 0)
            \end{gather*}
            \item Otherwise, if $e^2 \nmodn{8} D$, the HLK-invariant is also $(1, [2, 2, 0])$ since
            \begin{gather*}
    	        \fr(v_1) = (1/2, 1/2), \quad
    	        \fr(v_2) = (0, 1/2), \quad
            	\fr(v_3) = (1/2, 1/2), \\
            	\fr(v_4) = (0, 0)
            	\quad \text{ and } \quad
            	\fr(v_5) = (0, 1/2).
            \end{gather*}
        \end{enumerate}
        \item If $D \modn{8} 1$, we have two cases:
        \begin{enumerate}
            \item If $e \modn{4} 1$, the HLK-invariant is $(2, [1, 1, 1])$ since
            \begin{gather*}
    	        \fr(v_1) = (1/2, 1/2), \quad
    	        \fr(v_2) = (0, 1/2), \quad
            	\fr(v_3) = (1/2, 0), \\
            	\fr(v_4) = (0, 0)
            	\quad \text{ and } \quad
            	\fr(v_5) = (0, 0).
            \end{gather*}
            \item If $e \modn{4} 3$, the HLK-invariant is $(0, [3, 1, 1])$ since
            \begin{gather*}
    	        \fr(v_1) = (1/2, 1/2), \quad
    	        \fr(v_2) = (1/2, 1/2), \quad
            	\fr(v_3) = (1/2, 1/2), \\
            	\fr(v_4) = (1/2, 0)
            	\quad \text{ and } \quad
            	\fr(v_5) = (0, 1/2).
            \end{gather*}
        \end{enumerate}
    \end{enumerate}
    Observe that the HLK-invariants that we obtain are precisely the same as the ones found for the arithmetic case.
    
    The linear action on these points restricts the possibilities of the action of $\Aff(\modelL_D(e))$ on regular Weierstrass points. We get the following results which are completely analogous to the arithmetic case (that is, to \Cref{prop:arith-d2=0 mod4,prop:arith-d2=1 mod8}):
    \begin{enumerate}
        \item If $D$ is even, that is, $D \modn{8} 0$ or $D \modn{8} 4$, we have two cases:
        \begin{enumerate}
            \item If $e^2 \modn{8} D$, then
            \[
                G_D(e) \leqslant \DihFour{1}{3}{2}{4},
            \]
            \item Otherwise, if $e^2 \nmodn{8} D$, then
            \[
                G_D(e) \leqslant \DihFour{1}{2}{3}{5},
            \]
        \end{enumerate}
        \item If $D$ is odd, that is $D \modn{8} 1$, then
            \[
                G_D(e) \leqslant \Sym(\{1,2,3\}) \times \Sym(\{4,5\}).
            \]
    \end{enumerate}
    Thus, we get the following result:
    \begin{cor}\label{cor:upper-bounds}
    Let $D > 4$ be a real quadratic discriminant which is a quadratic residue modulo $8$. We have the following:
    \begin{enumerate}[label=(\alph*),beginpenalty=10000,endpenalty=10000]
        \item\label{ub:even} If $D \modn{4} 0$ and $X \in \Omega W_D$, then $G(X)$ is contained in a group conjugate to $\Dih_4$.
        \item\label{ub:odd} If $D \modn{8} 1$ and $X \in \Omega W_D$ (regardless of the spin), then $G(X)$ is contained in a group isomorphic to $\Sym_3\times\Sym_2 \cong \Dih_6$.
        \qed
    \end{enumerate}
    \end{cor}

\subsubsection{Quadratic nonresidues modulo~$8$}
\label{sec:quadratic-nonresidues_H(2)}

In the remaining case of $D \modn{8} 5$, let $e \in R_D$ and consider $\rho = \lambda = \frac{\sqrt{D} + e}{2}$.
With this choice of $\rho$, we get that the locationstm of the regular Weierstrass points (relative to the singularity) in $\modelL_D(e)$, up to addition by elements in $\O_D$, are
\begin{gather*}
	v_1 = \left(\frac{1}{2}, \frac{1}{2}\right),
	\quad
	v_2 = \left(\frac{1}{2}\rho, \frac{1}{2}\right),
	\quad
	v_3 = \left(\frac{1}{2}, \frac{1}{2}\rho\right), 
	\\
	v_4 = \left(\frac{1}{2}\rho, 0\right)
	\quad \text{ and } \quad
	v_5 = \left(0, \frac{1}{2}\rho\right).
\end{gather*}

\begin{lem}
Let $D > 4$ be a real quadratic discriminant with $D \modn{8} 5$. Let $e \in R_D$, and observe that $e$ is odd. Fix $\rho = \frac{\sqrt{D} + e}{2}$. If $A \in \SL(\modelL_D(e))$ fixes $v_1 \bmod \O_D$, then $\fr(Av_4) = \fr(Av_5)$.
\end{lem}

\begin{proof}
    Note first that
    \[
        \frac{1}{2}\rho^2 = \frac{e}{2}\rho + \frac{D-e^2}{8} \modn{\O_D} \frac{1}{2}\rho + \frac{1}{2}.
    \]
    Then, by \Cref{eq:Av}, we have that
    \begin{align*}
        A v &\modn{\O_D} \frac{1}{2}
        \begin{pmatrix}
            p_\alpha p_1 + p_\beta p_2 + q_\alpha q_1 + q_\beta q_2 \\
            p_\gamma p_1 + p_\delta p_2 + q_\gamma q_1 + q_\delta q_2
        \end{pmatrix} \\
            &\quad + \frac{\rho}{2}
        \begin{pmatrix}
            q_\alpha q_1 + q_\beta q_2 + p_\alpha q_1 + p_\beta q_2 + q_\alpha p_1 + q_\beta p_2 \\
            q_\gamma q_1 + q_\delta q_2 + p_\gamma q_1 + p_\delta q_2 + q_\gamma p_1 + q_\delta p_2
        \end{pmatrix}
    \end{align*}
    for any $v \in \frac{1}{2}(\O_D \times \O_D)$.
    In particular, we have that
        \[
        \fr(Av) \modn{1} \frac{1}{2}
    \begin{pmatrix}
        p_\alpha p_1 + p_\beta p_2 + q_\alpha q_1 + q_\beta q_2 \\
        p_\gamma p_1 + p_\delta p_2 + q_\gamma q_1 + q_\delta q_2
    \end{pmatrix}.
    \]
    
    Since $A$ fixes $v_1 \bmod{\O_D}$, we obtain that
    \[
        v_1 = \frac{1}{2}
        \begin{pmatrix}
            1 \\ 1
        \end{pmatrix}
        \modn{\O_D}
        \frac{1}{2}
        \begin{pmatrix}
            p_\alpha + p_\beta \\
            p_\gamma + p_\delta
        \end{pmatrix}
        +
        \frac{\rho}{2}
        \begin{pmatrix}
            q_\alpha + q_\beta \\
            q_\gamma + q_\delta
        \end{pmatrix}.
    \]
    It follows that 
    \begin{align*}
        p_\alpha + p_\beta & \modn{2} 1, &
        q_\alpha + q_\beta & \modn{2} 0, \\
        p_\gamma + p_\delta & \modn{2} 1, &
        q_\gamma + q_\delta & \modn{2} 0.
    \end{align*}
    
    Applying this to $v_4$ and $v_5$, we get
    \[
        \fr(Av_4)
        \modn{1}
        \frac{1}{2}
        \begin{pmatrix}
            q_\alpha \\
            q_\gamma
        \end{pmatrix}
        \modn{1}
        \frac{1}{2}
        \begin{pmatrix}
            q_\beta \\
            q_\delta
        \end{pmatrix}
        \modn{1}
        \fr(Av_5).
        \qedhere
    \] 
\end{proof}

\begin{prop}
    Let $D > 4$ be a real quadratic discriminant with $D \modn{8} 5$. Then, for any $e \in R_D$, the group $G_D(e) = G(\modelL_D(e))$ of permutations of the regular Weierstrass points of $\modelL_D(e)$ is such that
    \[
        G \leqslant \DihFive{1}{2}{5}{4}{3}.
    \]
\end{prop}
\begin{proof}
    Using $\rho = \frac{\sqrt{D}+e}{2}$ as in the previous lemma, we have that
    \begin{gather*}
        \fr(v_1) = (1/2, 1/2), \quad
        \fr(v_2) = (0, 1/2), \quad
    	\fr(v_3) = (1/2, 0), \\
    	\fr(v_4) = (0, 0)
    	\quad \text{ and } \quad
    	\fr(v_5) = (0, 0).
    \end{gather*}
    
    Let $A \in \SL(\modelL_D(e))$ that fixes $w_1$. Then, $A$ fixes $v_1 \bmod \O_D$ and, by the previous lemma, $\fr(Av_4) = \fr(A v_5)$. But $\{v_4,v_5\}$ is the only possible such pair, so $A$ fixes $\{w_4,w_5\}$. Since $A$ fixes $\{w_1\}$ and $\{w_4,w_5\}$, it also fixes $\{w_2,w_3\}$.
    
    In particular, every element in $G_D(e)$ that fixes $1$ has order at most~$2$.
    On the other hand, by \Cref{prop:D=5 mod8}, $G_D(e)$ has a subgroup of order~$10$.
    But, up to automorphisms, the subgroups of $\Sym_5$ of order divisible by $10$ are $\Dih_5$, $\GA(1,5)$, $A_5$ and $\Sym_5$. All of them, but $\Dih_5$, have elements of order greater than two that stabilise~$1$.
\end{proof}
We obtain the following straightforward corollary:
\begin{cor}
    \label{cor:upper-bounds-D=5 mod 8}
    Let $D > 4$ be a real quadratic discriminant with $D \modn{8} 5$ and $X \in \Omega W_D$.  Then, $G(X)$ is contained in a group conjugate to $\Dih_5$.
    \qed
\end{cor}

    \section{Small discriminants}
    \label{sec:remaining_H(2)}
    In this section we treat the remaining cases are not covered by our previous arguments. That is, the cases where the discriminant $D$ is either $9$ or $33$.

    \subsection{Discriminant 9}
    \label{sec:D=9} 
     \begin{figure}[t!]
            \centering
            \includegraphics[scale=1]{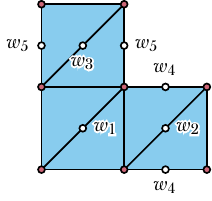}
            \caption{In the case of $D=9$, $\modelL_{9}(-1) \in \Omega W_{9}$ is a square-tiled surface and the diagonal direction induces the permutation $(4,5) \in G_{9}(-1)$.}
            \label{fig:9}
        \end{figure}
    We will start by fixing $D = 9$, so $R_{9} = \{-1\}$ and that the surface $S = \modelL_{25}(-1)$ is the primitive square-tiled surface depicted in \Cref{fig:9}.

        The action of the horizontal and vertical twists on $S$ generates $\Sym(\{1,2,3\}) \leqslant G(S)$, as was seen in \Cref{prop:D=1 mod8}.
        On the other hand, the diagonal direction, (that is, taking $\theta = \arctan(1) = \pi/4$) on $S$ gives a one-cylinder decomposition whose core curve contains $w_4$ and $w_5$, as shown in \Cref{fig:9}.
        It follows that $(4\;5) \in G(S)$ and, by \Cref{cor:upper-bounds}, we conclude that $G(S) = \Sym(\{1,2,3\})\times\Sym(\{4,5\}) \cong \Dih_6$.
        
        Finally, since $\Omega W_9$ is connected, we obtain that $G(X) \cong \Sym_3\times\Sym_2 \cong \Dih_6$ for every $X \in \Omega W_9$. Combining this with \Cref{prop:D=1 mod8}, we obtain that $G(X)$ can be generated by the action of three parabolic elements.
        \qed

    \subsection{Discriminant 33}
    \label{sec:D=33}
        We now focus our attention to the case where $D = 33$. We have that $R_{33} = \{-5,-3,-1,1\}$ and we will consider the surfaces $\modelL_{33}(-1)$ and $\modelL_{33}(1)$, depicted in \Cref{fig:33}. By \Cref{eq:spin}, these surfaces belong to distinct components of $\Omega W_{33}$.
        \begin{figure}[t!]
            \centering
            \begin{subfigure}[b]{0.32\textwidth}
                \centering
                \includegraphics[scale=1]{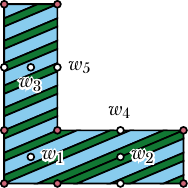}
                \caption*{Slope $\frac{1}{\lambda-e-1} = \frac{1}{\lambda}$ on the surface $\modelL_{33}(-1)$.}
                \label{fig:33--1}
            \end{subfigure}
            \qquad
            \begin{subfigure}[b]{0.32\textwidth}
                \centering
                \includegraphics[scale=1]{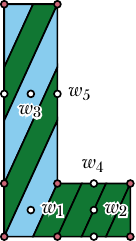}
                \caption*{Slope $\frac{3}{\lambda-e-1} = \frac{3}{\lambda-2}$ on the surface $\modelL_{33}(1)$.}
                \label{fig:33-1}
            \end{subfigure}
            \caption{In the case $D=33$, we exhibit explicit directions that induce the permutation $(4,5) \in G_{33}(e)$, for $e = -1$ (left) and $e = 1$ (right). The cylinder $C_1$ is shown in a darker shade, while $C_2$ is depicted in a lighter shade.}
            \label{fig:33}
        \end{figure}
        
        By \Cref{prop:D=1 mod8}, the action of the horizontal and vertical twists on both surfaces generates $\Sym_3 \leqslant G(S)$. In both cases, we will exhibit an explicit direction that induces the remaining transposition $(4\;5)$. After this is done, we will obtain that $G(X)$ is isomorphic to $\Sym_3 \times \Sym_2 \cong \Dih_6$ for every $X \in \Omega W_{33}$. Combining this with \Cref{prop:D=1 mod8}, we obtain that $G(X)$ can be generated by the action of three parabolic elements.
        
        To simplify the computations, we define the \emph{horizontal height} of a cylinder to be the length of a maximal horizontal segment contained in the cylinder. Indeed, when computing the ratio of the moduli of two cylinders in a cylinder decomposition, we may replace the height by the horizontal height and obtain the same result. In fact, the horizontal height is equal to the height up to a linear factor that is cancelled out when taking the ratio. In other words, the horizontal height corresponds to the actual height of the same cylinder when considering a sheared version of the surface; the shear is performed in the direction of the cylinder decomposition and, thus, does not change any cylinder width. By a slight abuse of notation, we will denote the horizontal heights by $h$ as we previously did for the actual heights.
        
        Similarly, we will also use the \emph{horizontal width} of a cylinder, which is the length of the projection of the cylinder to the horizontal axis. An argument analogous to the paragraph above shows that, when computing ratio of moduli, using the actual width of the horizontal width makes no difference. We denote the horizontal width by $w$.
        
        First consider the case $\modelL_{33}(-1)$ as in \Cref{fig:33--1}. We take the cylinder decomposition with slope $\frac{1}{\lambda}$, that is, in direction $\theta = \arctan\left(\frac{1}{\lambda}\right)$. 
        
        Let $C_1$ be the cylinder containing the points $w_4$ and $w_5$, and let $C_2$ be the other cylinder of the decomposition. We need to show that the ratio of the moduli of these cylinders belongs to $\frac{\mathrm{odd}}{\mathrm{odd}}$ or $\frac{\mathrm{even}}{\mathrm{odd}}$ to show that the transposition $(4\;5)$ belongs to $G(\modelL_{33}(-1))$.
        
        As already discussed, we will use horizontal heights instead of heights to compute the moduli. We denote these heights by $h_1$ and $h_2$, respectively. We have that these numbers satisfy the equations
        \begin{align*}
            h_1 + h_2 &= 1 \\
            3 h_1 + 2 h_2 &= \lambda.
        \end{align*}
        so $h_1 = \lambda - 2$ and $h_2 = 3 - \lambda$.
        
        The modulus of $C_1$ is then
        \[
            m_1 = \frac{ \lambda + 1 + 1 + 1 + 1 + 1 + h_2 + \lambda + h_1 + (\lambda - h_1) + (\lambda - 1) + 1}{h_1} = \frac{3\lambda + 8}{\lambda - 2},
        \]
        and the modulus of $C_2$ is then
        \[
            m_2 = \frac{ \lambda + (h_1 + h_2) + (\lambda - (h_1+h_2)) + 2(h_1+h_2) + 1 + 1 + 1 + 1 + 1 + 1}{h_2} = \frac{2\lambda + 8}{\lambda - 3}.
        \]
        
        We obtain that $r_\theta = \frac{m_1}{m_2} = \frac{-3\lambda^2 + \lambda + 24}{2 (\lambda^2 + 2\lambda - 8)}$. Since $\lambda = \frac{\sqrt{33} - 1}{2}$, an elementary computation shows that $r_\theta = 2$. Thus, $\tau_\theta = (4\;5) \in G(\modelL_{33}(-1))$. Combining this with \Cref{prop:D=1 mod8}, we obtain that $G(\modelL_{33}(-1)) \cong \Sym_3 \times \Sym_2 \cong \Dih_6$.
        
        We now focus our attention on $\modelL_{33}(1)$ and consider the cylinder decomposition with slope $\frac{3}{\lambda - 2}$, that is, with direction $\theta = \arctan\left(\frac{3}{\lambda - 2}\right)$.
        
        Let $C_1$ be the cylinder containing the points $w_4$ and $w_5$, and let $C_2$ be the other cylinder of the decomposition. As in the previous case, we need to show that the ratio of the moduli of these cylinders has odd denominator (when written as an irreducible fraction) to obtain that $(4\;5) \in G(\modelL_{33}(1))$. Let $h_1$ and $h_2$ be the horizontal heights of $C_1$ and $C_2$, respectively. We have that
        \begin{align*}
            h_1 + h_2 &= 1 \\
            3h_1 &= \lambda - 2.
        \end{align*}
        so $h_1 = \frac{\lambda - 2}{3}$ and $h_2 = \frac{5 - \lambda}{3}$.
        
        The modulus of $C_1$ is then
        \[
            m_1 = \frac{h_1 + h_1 + h_1 + 1 + 1 - h_1 + h_1}{h_1} = \frac{3\lambda}{\lambda - 2},
        \]
        and that the modulus of $C_2$ is then
        \[
            m_2 = \frac{1 + 1}{h_2} = \frac{6}{5 - \lambda}.
        \]
        We obtain that $r_\theta = \frac{m_1}{m_2} = \frac{\lambda(5 - \lambda)}{2(\lambda - 2)}$. Since $\lambda = \frac{\sqrt{33} + 1}{2}$, an elementary computation shows that $r_\theta = 2$. Thus, $\tau_\theta = (4\;5) \in G(\modelL_{33}(-1))$. Combining this with \Cref{prop:D=1 mod8}, we obtain that $G(\modelL_{33}(-1)) \cong \Sym_3 \times \Sym_2 \cong \Dih_6$.
        \qed

    \section{Classification of permutation groups}
    \label{sec:proof_H(2)}
    
    In this section we collect our results to give a complete proof of \Cref{thm:main_H(2)} and \Cref{cor:main_H(2)}.
    
    \begin{proof}[Proof of \Cref{thm:main_H(2)}]
    In the case $D \modn{4} 0$, \Cref{prop:D=0 mod4} shows a first containment, namely that $G(X)$ contains a subgroup conjugate to $\Dih_4$ for any $X \in \Omega W_D$. Moreover,
    \Cref{prop:arith-d2=0 mod4} gives the reverse inclusion in the arithmetic case,
    and \Cref{cor:upper-bounds}\ref{ub:even} provides the corresponding reverse inclusion in the non-arithmetic case.
    
    In the case $D \modn{8} 1$, \Cref{prop:arith-d2=1 mod8} and \Cref{cor:upper-bounds}\ref{ub:odd} show a first containment, that is, that $G(X)$ is contained in a group conjugate to $\Sym_3 \times \Sym_2 \cong \Dih_6$, for the arithmetic and non-arithmetic cases, respectively. The reverse inclusion can be deduced \emph{a posteriori} from \Cref{prop:D=1 mod8,prop:proper}, for $D \notin \{9,33\}$. Indeed, the former proposition shows that $G(X)$ contains a subgroup conjugate to $\Sym_3$, while the latter shows that it contains a product of two disjoint transpositions. Since no product of disjoint two transpositions belongs to $\Sym_3$, we conclude that $G(X)$ contains a \emph{proper} subgroup conjugate to $\Sym_3$. Thus, $G(X)$ must be isomorphic to $\Sym_3 \times \Sym_2 \cong \Dih_6$.
    
    In the case $D \modn{8} 5$, \Cref{prop:D=5 mod8} shows the first containment, namely that $G(X)$ contains a subgroup conjugate to $\Dih_5$ for any $X \in \Omega W_D$. Finally, \Cref{cor:upper-bounds-D=5 mod 8} gives the corresponding reverse inclusion.
    
    The only remaining cases are the exceptional discriminants $D = 9$ and $D = 33$. These cases are treated in \Cref{sec:D=9} and \Cref{sec:D=33}, respectively.
    \end{proof}
    
    \begin{proof}[Proof of \Cref{cor:main_H(2)}]
        For $D \modn{4} 0$ and $D \modn{8} 5$, we have shown this result in \Cref{prop:D=0 mod4,prop:D=5 mod8}.
        
        In the case of $D \modn{8} 1$ and $D \notin \{9,33\}$, we have by \Cref{prop:proper} that $G(X)$ contains a product of two disjoint transpositions induced by the action of a \emph{single} parabolic element. Moreover, \Cref{prop:D=1 mod8} shows that $G(X)$ contains a subgroup conjugate to $\Sym_3$, which is, in turn, generated by the action of two parabolic elements. We known that these three parabolic elements generate $G(X)$, but one of these elements is redundant. Indeed, assume that we label the regular Weierstrass points in such a way that $G(X)$ is contained in $\Sym(\{1, 2, 3\}) \times \Sym(\{4, 5\})$. Then, the product of two disjoint transpositions must be of the form $\tau = (r\; s)(4\;5)$, where $r < s$ and $r, s \in \{1, 2, 3\}$. Since there exist two parabolic elements in $\SL(X)$ whose actions generate a subgroup of $G(X)$ conjugate to $\Sym(\{1, 2, 3\})$, some parabolic element gives a transposition $\sigma \neq (r\;s)$. We have that $\langle \sigma, \tau\rangle = \Sym(\{1, 2, 3\}) \times \Sym(\{4, 5\})$, so we obtain the desired result.
        
        For $D \in \{9,33\}$, no parabolic element can produce a product of disjoint transpositions (see \Cref{rem:D=9/33}). In particular, since no pair of transpositions can generate a group conjugate to $\Sym(\{1,2,3\})\times\Sym(\{4,5\})$, there is no pair of parabolic elements that can generate the corresponding permutation group.
        On the other hand, we show in \Cref{sec:D=9} that $G(X)$ can be generated by the action of three parabolic elements for any $X \in \Omega W_9$ and, similarly, we show in \Cref{sec:D=33} that $G(X)$ can be generated by the action of three parabolic elements for any $X \in \Omega W_{33}$.
    \end{proof}

\sloppy
\printbibliography

\end{document}